\documentclass[10pt,twoside,reqno]{amsart}
\usepackage{color}
\usepackage{graphicx}
\usepackage{amsfonts,amsmath, amssymb,mathtools}
\usepackage{amsrefs}
\usepackage{todonotes}
\allowdisplaybreaks

\usepackage{enumitem}
\usepackage{comment}
\usepackage{relsize}
\usepackage{stackengine,wasysym}
\usepackage{enumitem}
\usepackage{cite}
\usepackage{xcolor}
\usepackage{hyperref}

\usepackage{amsmath,amssymb}
\usepackage{mathrsfs}
\usepackage{cases}
\usepackage[toc,page]{appendix}
\allowdisplaybreaks
\theoremstyle{plain}
\newtheorem{lemma}{Lemma}[section]
\newtheorem{theorem}[lemma]{Theorem}

\newtheorem{definition}[lemma]{Definition}

\theoremstyle{remark}
\newtheorem{remark}{Remark}

\newcommand*  {\D}{{\mathcal D}}

\newcommand*{\abs}[1]{\lvert #1 \rvert}

\newcommand{\Et}{\mathcal{E}_{T}} 
\newcommand{\Eg}{\mathcal{E}_G} 

\makeatletter
\newcommand*{\rom}[1]{\expandafter\@slowromancap\romannumeral #1@}
\makeatother

\def\p{\partial}

\numberwithin{equation}{section}

\usepackage{mathtools}

\begin{document}

\title[PINNs for Primitive equations]{Higher-Order Error estimates for physics-informed neural networks approximating the primitive equations}

\author[R. Hu]{Ruimeng Hu}
\address[R. Hu]
{	Department of Mathematics \\
Department of Statistics and Applied Probability \\
     University of California  \\
	Santa Barbara, CA 93106, USA.} 
	\email{rhu@ucsb.edu}

\author[Q. Lin]{Quyuan Lin*}\thanks{*Corresponding author. Department of Mathematics, University of California,	Santa Barbara, CA 93106, USA. E-mail address: quyuan\_lin@ucsb.edu}
\address[Q. Lin]
{	Department of Mathematics \\
     University of California  \\
	Santa Barbara, CA 93106, USA.} \email{quyuan\_lin@ucsb.edu}
	
\author[A. Raydan]{Alan Raydan}
\address[A. Raydan]
{	Department of Mathematics \\
     University of California  \\
	Santa Barbara, CA 93106, USA.} \email{alanraydan@ucsb.edu}

\author[S. Tang]{Sui Tang}
\address[S. Tang]
{	Department of Mathematics \\
     University of California  \\
	Santa Barbara, CA 93106, USA.}
\email{suitang@ucsb.edu}

\date{\today}

\begin{abstract}
Large-scale dynamics of the oceans and the atmosphere are governed by primitive equations (PEs). Due to the nonlinearity and nonlocality, the numerical study of the PEs is generally challenging. Neural networks have been shown to be a promising machine learning tool to tackle this challenge. In this work, we employ physics-informed neural networks (PINNs) to approximate the solutions to the PEs and study the error estimates. We first establish the higher-order regularity for the global solutions to the PEs with either full viscosity and diffusivity, or with only the horizontal ones. Such a result for the case with only the horizontal ones is new and required in the analysis under the PINNs framework. Then we prove the existence of two-layer tanh PINNs of which the corresponding training error can be arbitrarily small by taking the width of PINNs to be sufficiently wide, and the error between the true solution and its approximation can be arbitrarily small provided that the training error is small enough and the sample set is large enough. In particular, all the estimates are \textit{a priori}, and our analysis includes higher-order (in spatial Sobolev norm) error estimates. 
Numerical results on prototype systems are presented to further illustrate the advantage of using the $H^s$ norm during the training.
\end{abstract}

\maketitle




\section{Introduction}

\subsection{The Primitive Equations}

The study of global weather prediction and climate dynamics is largely dependent on the atmosphere and oceans. Ocean currents transport warm water from low latitudes to higher latitudes, where the heat can be released into the atmosphere to balance the earth's temperature. A widely accepted model to describe the motion and state of the atmosphere and ocean is the Boussinesq system, a combination of the Navier–Stokes equations (NSE) with rotation and a heat (or salinity) transport equation. As a result of the extraordinary organization and complexity of the flow in the atmosphere and ocean, the full governing equations appear to be difficult to analyze, at least for the foreseeable future. In particular, the global existence and uniqueness of the smooth solution to the 3D NSE is one of the most challenging mathematical problems. 

Fortunately, when studying oceanic and atmospheric dynamics at the planetary scale, the vertical scale (a few kilometers for the ocean, 10-20 kilometers for the atmosphere) is much smaller than the horizontal scale (many thousands of kilometers). Accordingly, the large-scale ocean and atmosphere satisfy the hydrostatic balance based on scale analysis, meteorological observations, and historical data. By virtue of this, the primitive equations (PEs, also called the hydrostatic Navier-Stokes equations) are derived as the asymptotic limit of the small aspect ratio between the vertical and horizontal length scales from the Boussinesq system \cites{azerad2001mathematical,li2019primitive,li2022primitive,furukawa2020rigorous}. Because of the impressive accuracy, the following $3D$ viscous PEs is a widely used model in geophysics (see, e.g., \cites{blumen1972geostrophic,gill1976adjustment,gill1982atmosphere,hermann1993energetics,holton1973introduction,kuo1997time,plougonven2005lagrangian,rossby1938mutual} and references therein):
\begin{subequations}\label{PE-system}
\begin{align}
    &\partial_t V + V\cdot \nabla_h  V + w\partial_z V -\nu_h \Delta_h V - \nu_z \partial_{zz} V +f_0 V^\perp + \nabla_h  p = 0 , \label{PE-1}
    \\
    &\partial_z p + T= 0, \label{PE-2}
    \\
    &\nabla_h  \cdot V + \partial_z w =0, \label{PE-3}
    \\
    &\partial_t T + V\cdot \nabla_h  T + w\partial_z T -\kappa_h \Delta_h T - \kappa_h \partial_{zz} T = Q \label{PE-4}.
\end{align}
\end{subequations}
Here the horizontal velocity $V = (u, v)$, vertical velocity $w$, the pressure $p$, and the temperature $T$ are the unknown quantities which are functions of the time and space variables $(t, x, y, z)$. The $2D$ horizontal gradient and Laplacian are denoted by $\nabla_h = (\partial_{x}, \partial_{y})$ and $\Delta_h = \partial_{xx} + \partial_{yy}$, respectively. The nonnegative constants $\nu_h, \nu_z, \kappa_h$ and $\kappa_z$ are the horizontal viscosity, the vertical viscosity, the horizontal diffusivity and the vertical diffusivity coefficients, respectively. The parameter $f_0 \in \mathbb{R}$ stands for the Coriolis parameter, $Q$ is a given heat source, and the notation $V^\perp = (-v,u)$ is used.

According to whether the system has horizontal or vertical viscosity, there are mainly four different models considered in the literature (some works also consider the anisotropic diffusivity).
\begin{enumerate}[ label=\textbf{Case \arabic*}]
    \item\label{Case1} PEs with full viscosity, i.e., $\nu_h>0, \nu_z>0$: The global well-posedness of strong solutions in Sobolev spaces was first established in \cites{cao2007global}, and later in \cites{kobelkov2006existence}; see also the subsequent articles \cites{kukavica2007regularity} for different boundary conditions,  as well as \cites{hieber2016global} for some progress towards relaxing the smoothness on the initial data by using the semigroup method. 
    
    \item\label{Case2} PEs with only horizontal viscosity, i.e., $\nu_h>0, \nu_z=0$: \cites{cao2016global,cao2017strong,cao2020global} consider horizontally viscous PEs with anisotropic diffusivity and establish global well-posedness.
    
    \item\label{Case3} PEs with only vertical viscosity, i.e., $\nu_h=0, \nu_z>0$: Without the horizontal viscosity, PEs are shown to be ill-posed in Sobolev spaces \cites{renardy2009ill}. In order to get well-posedness, one can consider some additional weak dissipation \cites{cao2020well}, or assume the initial data have Gevrey regularity and be convex \cites{gerard2020well}, or be analytic in the horizontal direction \cites{paicu2020hydrostatic,lin2022effect}. It is worth mentioning that whether smooth solutions exist globally or form singularity in finite time is still open.
    
    \item\label{Case4} Inviscid PEs, i.e., $\nu_h=0, \nu_z=0$: The inviscid PEs are ill-posed in Sobolev spaces \cites{renardy2009ill,han2016ill,ibrahim2021finite}. Moreover, smooth solutions of the inviscid PEs can form singularity in finite time \cites{cao2015finite,wong2015blowup,ibrahim2021finite,collot2021stable}. On the other hand, with either some special structures (local Rayleigh condition) on the initial data in $2D$, or real analyticity in all directions for general initial data in both $2D$ and $3D$, the local well-posedness can be achieved \cites{brenier1999homogeneous,brenier2003remarks,ghoul2022effect,grenier1999derivation,kukavica2011local,kukavica2014local,masmoudi2012h}.
\end{enumerate}

Others also consider stochastic PEs, that is, the system \eqref{PE-system} with additional random external forcing terms (usually characterized by generalized Wiener processes on a Hilbert space). For existence and uniqueness of solutions to those systems, see \cites{glatt2008stochastic,glatt2011pathwise,brzezniak2021well,debussche2011local,debussche2012global,saal2021stochastic,hu2022local,hieber2020primitive,slavik2021large,hu2023pathwise}.

In this paper, we focus on \ref{Case1} and \ref{Case2} in which the well-posedness is established in Sobolev spaces. \ref{Case1} is also assumed to have full diffusivity, while \ref{Case2} is considered to have only horizontal diffusivity. The analysis of the \ref{Case3} and \ref{Case4} requires rather different techniques as those models are ill-posed in Sobolev spaces for general initial data, and are left for future work.  

System \eqref{PE-system} has been studied under some proper boundary conditions. For example, as introduced in \cites{cao2016global,ghoul2022effect},
we consider the domain to be $\mathcal D := \mathcal M \times (0,1)$ with $\mathcal M := (0,1)\times (0,1)$ and
\begin{equation}\label{BC-T3}
\begin{split}
     &V, w, p, T \text{ are periodic in }  (x,y,z) \text{ with period }  1, 
     \\
      &V \text{ and } p  \text{ are even in }  z,  \text{ and }  w \text{ and } T \text{ are odd in }  z.
\end{split}
\end{equation}
Note that the space of periodic functions with such symmetry condition is invariant under the dynamics of system \eqref{PE-system}, provided that $Q$ is periodic in $(x,y,z)$ and odd in $z$. When system \eqref{PE-system} is considered in $2D$ space, the system will be independent of the $y$ variable. In addition to the boundary condition, one  needs to impose the initial condition
\begin{equation}\label{IC}
    (V,T)|_{t=0} = (V_0,T_0).
\end{equation}

We point out that there is no initial condition for $w$ since $w$ is a diagnostic variable and it can be written in terms of $V$ (see \eqref{w}). This is different from the Navier-Stokes equations and Boussinesq system.

\subsection{PINNs} Due to the nonlinearity and nonlocality of many PDEs (including PEs), the numerical study for them is in general a hard task. A non-exhaustive list of the numerical study of PEs includes \cites{shen1999fast,chen2012numerical,bousquet2020numerical,samelson2003surface,liu2008fourth,charney1955use,chen2003unstructured,smagorinsky1963general,krishnamupti2018introduction,pei2018continuous,korn2021strong} and references therein. In the past few years, the deep neural network has emerged as a promising alternative, but it requires abundant data that cannot always be found in scientific research. Instead, such networks can be trained from additional information obtained by enforcing the physical laws. 	Physics-informed machine learning seamlessly integrates data and mathematical physics models, even in partially understood, uncertain, and high-dimensional contexts \cites{karniadakis2021physics}.

Recently, physics-informed neural networks (PINNs) have been shown as an efficient tool in scientific computing and in solving challenging PDEs. 
PINNs, which approximate solutions to PDEs by training neural networks to minimize the residuals coming from the initial conditions, the boundary conditions, and the PDE operators, have gained a lot of attention and have been studied intensively. The study of PINNs can be retrieved back to the 90s \cites{dissanayake1994neural,lagaris1998artificial,lagaris2000neural}. Very recently,  \cites{raissi2018hidden,raissi2019physics} introduced and illustrated the PINNs approach for solving nonlinear PDEs, which can handle both forward problems and inverse problems. For a much more complete list of recent advances in the study of PINNs, we refer the readers to \cites{cuomo2022scientific,karniadakis2021physics} and references therein. We also remark that the deep Galerkin method \cites{sirignano2018dgm} shares a similar spirit with PINNs.

In addition to investigating the efficiency and accuracy of PINNs in solving PDEs numerically, researchers are also interested in rigorously evaluating the error estimates. In a series of works \cites{mishra2022estimates1,mishra2022estimates,de2021error,de2022error}, the authors studied the error analysis of PINNs for approximating several different types of PDEs. It is worth mentioning that, recently, there has been a result on generic bounds for PINNs established in \cites{de2022generic}. We remark that our work is devoted to establishing higher-order error estimates based on the higher-order regularity of the solutions to the PEs. The analysis requires nontrivial efforts and techniques due to the special characteristics of the PEs, and these results are not trivially followed from \cites{de2022generic}.

To set up the PINNs framework for our problem, we first review the PINN algorithm \cites{raissi2018hidden,raissi2019physics} for a generic PDE with initial and boundary conditions: for $x\in \mathcal D, y\in \partial \mathcal D, t\in[0,\mathcal T]$, the solution $u$ satisfies
\begin{align*}
    \text{PDE operator: }&\mathcal D[u](x,t)=0,
    \\
    \text{Initial condition: } &u(x,0)=\phi(x),
    \\
    \text{Boundary condition: }&\mathcal Bu(y,t) = \psi(y,t).
\end{align*}
The goal is to seek a neural network $u_\theta$ where $\theta$ represents all network parameters (see Definition \ref{def:nn} for details) so that 
\begin{align*}
    \text{PDE residual: }&\mathcal R_i[\theta](x,t) = \mathcal D[u_\theta](x,t),
    \\
    \text{Initial residual: }&\mathcal R_t[\theta](x,t) = u_\theta(x,0)-\phi(x), 
    \\
    \text{Boundary residual: }&\mathcal R_b[\theta] = \mathcal Bu_\theta(y,t) - \psi(y,t),
\end{align*}
are all small. Based on these residuals, for $s\in\mathbb N$ we defined the \textit{generalization error} for $u_\theta$:
\begin{equation} \label{eq:originalPINN}
\Eg[s;\theta]^2 = \int_0^T (\|\mathcal R_i\|_{H^s( \mathcal D)}^2 + \|\mathcal R_b\|_{H^s(\partial\mathcal D)}^2) dt + \|\mathcal R_t\|_{H^s(\mathcal D)}^2.
\end{equation}
 Notice that when $u_\theta=u$ is the exact solution, all the residuals will be zero and thus $\Eg[s;\theta]=0$. In practice, one uses numerical quadrature to approximate the integral appearing in $\Eg[s;\theta]$. We call the corresponding numerical quadrature the {\it training error} $\Et[s;\theta;\mathcal S]$ (see Section \ref{sec:quadrature} for details), which is also the \textit{loss function} used during the training. The terminology ``physics-informed neural networks" is used in the sense that the physical laws coming from the PDE operator and initial and boundary conditions lead to the residuals, which in turn give the generalization error and training error (loss function). Finally, the neural networks minimize the loss function during the training to obtain the approximation for the PDE. Note that in the literature, the analysis for PINN algorithms exists only for $s = 1$.

For our problem, residuals are defined in \eqref{residuals-pde}--\eqref{residuals-boundary}, the generalization error is defined in \eqref{generalization-error}, and training error is defined in \eqref{training-error}. In order to measure how well the approximation $u_\theta$ is, we use the {\it total error} $\mathcal E[s;\theta]^2 = \int_0^t \|u-u_\theta\|_{H^s}^2 dt$, and it is defined in \eqref{total-error} for our problem. And our analysis is for any $s \in \mathbb N$.

In this work, we mainly want to answer two crucial questions concerning the reliability of PINNs:
\begin{enumerate}[ label=\textbf{Q\arabic*}]
    \item\label{Q1}\hspace{-5pt}: The existence of neural networks $(V_\theta,w_\theta,p_\theta,T_\theta)$ such that the training error (loss function) $\Et[s;\theta;\mathcal S]<\epsilon$ for arbitrary $\epsilon>0$;
    \item\label{Q2}\hspace{-5pt}: \hspace{-1pt}The control of total error $\mathcal E[s;\theta]$ by the training error with large enough sample set $\mathcal S$, i.e., $\mathcal E[s;\theta]\lesssim \Et[s;\theta;\mathcal S] + f(|\mathcal S|)$ for some function $f$ which is small when $|\mathcal S|$ is large.
\end{enumerate}
An affirmative answer to \ref{Q1} implies that one is able to train the neural networks to obtain a small enough training error (loss function) at the end. An affirmative answer to \ref{Q2} guarantees that $(V_\theta,w_\theta,p_\theta,T_\theta)$ can approximate the true solution arbitrarily well in some Sobolev norms as long as the training error $\Et[s;\theta;\mathcal S]$ is small enough and the sample set $\mathcal S$ is large enough. However, $\Et[s;\theta;\mathcal S]$ is not convenient in the analysis, while $\Eg[s;\theta]$ provides a better way as it is in the integral form. As discussed in \cites{de2022error}, one can, in turn consider the following three sub-questions:
\begin{enumerate}[ label=\textbf{SubQ\arabic*}]
    \item\label{SubQ1}\hspace{-4pt}: The existence of neural networks $(V_\theta,w_\theta,p_\theta,T_\theta)$ such that the generalization error $\Eg[s;\theta]<\epsilon$ for arbitrary $\epsilon>0$; 
    \item\label{SubQ2}\hspace{-4pt}: The control of total error by generalization error, i.e., $\mathcal E[s;\theta]\lesssim \Eg[s;\theta]$;
    \item\label{SubQ3}\hspace{-4pt}: \hspace{-1pt}The difference between the generalization error and the training error $\Big|\Eg[s;\theta]- \Et[s;\theta;\mathcal S]\Big|< f(|\mathcal S|)$ for some function $f$ which is small when $|\mathcal S|$ is large.
\end{enumerate}
Specifically, the answers of \ref{SubQ1} and \ref{SubQ3} lead to the positive answer of \ref{Q1}, and the answers of \ref{SubQ2} and \ref{SubQ3} give the solution to \ref{Q2}.

 Our main contributions and results in this work are the followings:
\begin{itemize}
    \item We establish the higher-order regularity result for the solutions to the PEs under \ref{Case1} and \ref{Case2}, see Theorem \ref{thm:higher-regularity}. To our best knowledge, such a result for \ref{Case1} was proven in \cites{ju2020}, but is new for \ref{Case2}. It is necessary as the smoothness of the solutions is required in order to perform higher-order error analysis for PINNs.
    \item We answer \ref{Q1} and \ref{Q2} (and \ref{SubQ1}--\ref{SubQ3}) for the PINNs approximating the solutions to the PEs, which shows the PINNs is a reliable numerical method for solving PEs, see Theorems \ref{thm:generalization-error}, \ref{thm:total-error}, \ref{theorem:ge-by-tr}, \ref{thm:main}. Our estimates are all \textit{a priori}, and the key strategy is to introduce a penalty term in the generalization error \eqref{generalization-error} and the training error \eqref{training-error}. The introduction of such penalty terms is inspired by \cites{biswas2022error}, where the authors studied the PINNs approximating the $2D$ NSE. By virtue of Theorem \ref{thm:higher-regularity}, the solutions for PEs in \ref{Case1} and \ref{Case2} exist globally, and therefore are bounded for any finite time. Such penalty terms is introduced to control the growth of the outputs of neural networks and to make sure they are in the target bounded set. 
    
    \item Rather than just consider $L^2$ norm in the errors \cites{de2022error,mishra2022estimates}, i.e., $\mathcal E[s;\theta], \Eg[s;\theta], \Et[s;\theta;\mathcal S]$ with $s=0$, we use higher-order $H^s$ norm in $\mathcal E[s;\theta], \Eg[s;\theta], \Et[s;\theta;\mathcal S]$ for $s\in \mathbb N$. We prove that the usage of $H^s$ norm in $\Et[s;\theta;\mathcal S]$ will guarantee the control for $\mathcal E[s;\theta]$ with the same order $s$. The numerical performance in Section \ref{sec:numerical} further verifies our theory. Such results are crucial, as some problems do require higher-order estimates, for example, the Hamilton-Jacobi-Bellman equation requires the $L^p$ error estimate with $p$ large enough in order to be stable, see \cites{wang20222}. We refer the readers to \cites{czarnecki2017sobolev} for more discuss on the higher order error estimates for neural networks. We believe that the higher-order error estimates developed in this work can be readily applied to other PDEs, for example, the Euler equations, the Navier-Stokes equations, and the Boussinesq system.  
\end{itemize}

The rest of the paper is organized as the following. In Section \ref{sec:preliminaries}, we introduce the notation and collect some preliminary results that will be used in this paper. In Section \ref{sec:PE-regularity}, we prove the higher-order regularity of the solutions to the PEs under \ref{Case1} and \ref{Case2}. In Section \ref{sec:error-estimate}, we establish the main results of this paper by answering \ref{Q1} and \ref{Q2} (through \ref{SubQ1}--\ref{SubQ3}) discussed above. In the end, we present some numerical experiments in Section \ref{sec:numerical} to support our theoretical results on the accuracy of the approximation under higher-order Sobolev norms.

\section{Preliminaries}\label{sec:preliminaries}
In this section, we introduce the notation and collect some preliminary results that will be used in the rest of this paper. The universal constant $C$ that appears below may change from step to step, and we use the subscript to emphasize its dependence when necessary, e.g., $C_r$ is  a constant depending only on $r$. 

\subsection{Functional Settings}
We use the notation $\boldsymbol{x}:= (\boldsymbol{x}',z) = (x, y, z)\in \mathcal D$, where $\boldsymbol{x}'$ and $z$ represent the horizontal and vertical variables, respectively, and for a positive time $\mathcal T>0$ we denote by 
$$
\Omega =  [0,\mathcal T]\times \mathcal D .
$$
Let $\nabla=(\p_x,\p_y,\p_z)$ and $\Delta = \p_{xx}+\p_{yy}+\p_{zz}$ be the three dimensional gradient and Laplacian, and $\nabla_h = (\partial_{x}, \partial_{y})$ and $\Delta_h = \partial_{xx} + \partial_{yy}$ be the horizontal ones.
Let $\alpha \in \mathbb{N}^n$ be a multi-index. We say $\alpha\leq \beta$ if and only if $\alpha_i\leq \beta_i$ for each $i\in\{1,2,...,n\}$. The notation 
\begin{eqnarray*}
 |\alpha| = \sum\limits_{j=1}^n \alpha_j, \hspace{0.2in} \alpha! = \prod\limits_{j=1}^n \alpha_j!, \hspace{0.2in}  \binom{\alpha}{\beta} = \frac{\alpha!}{\beta!(\alpha-\beta)!}
\end{eqnarray*}
will be used throughout the paper. Let $P_{m,n}=\{\alpha\in\mathbb N^n, |\alpha|=m \}$, and denote by $|P_{m,n}|$ the cardinality of $P_{m,n}$, which is given by $|P_{m,n}| = \binom{m+n-1}{m}$.
For a function $f$ defined on an open subset $U\subseteq\mathbb R^n$ and $x=(x_1,x_2,...,x_n)\in U$, we denote  the  partial derivative of $f$ with multi-index $\alpha$ by
$
D^\alpha f= \frac{\partial^{|\alpha|}f}{\partial_{x_1}^{\alpha_1}\cdots\partial_{x_n}^{\alpha_n}}.
$
Let 
$$L^2(U)=\left\{f: \int_U |f(x)|^2 dx < \infty \right\}$$ 
be the usual $L^2$ space associated with the Lebesgue measure restricted on $U$,
endowed with the norm 
$
    \|f\|_{L^2(U)} = (\int_{U} |f|^2 dx)^{\frac{1}{2}},
$
coming from the inner product
$
    \langle f,g\rangle = \int_{U} f(x)g(x) \;dx
$
for $f,g \in L^2(U)$. For $r\in \mathbb N$, denote by $H^r(U)$ the Sobolev spaces:
$$
H^r(U) = \{f\in L^2(U) : \|D^\alpha f\|_{L^2(U)}<\infty \text{ for } |\alpha|\leq r\}, 
$$
endowed with the norm 
$
    \|f\|_{H^r(U)} = (\sum\limits_{|\alpha|\leq s}\int_{U} |D^\alpha f|^2 dx)^{\frac{1}{2}}.
$
For more details about the Sobolev spaces, we refer the readers to \cites{adams2003sobolev}. Define
\begin{equation*}
\begin{split}
     &\widetilde V_e := \left\{\varphi\in C^\infty(\mathcal D): \varphi \text{ is periodic in } (x,y,z) \text{ and even in } z, \, \int_0^1 \nabla_h\cdot \varphi (\boldsymbol x', z)dz = 0  \right\},
     \\
     &\widetilde V_o := \left\{\varphi\in C^\infty(\mathcal D): \varphi \text{ is periodic in } (x,y,z) \text{ and odd in }z  \right\},
\end{split}
\end{equation*}
and denote by $H^r_e(\mathcal D)$ and $H^r_o(\mathcal D)$  the closure spaces of $\widetilde V_e$ and $\widetilde V_o$, respectively, under the $H^r$-topology. When $r=0$, $H^r_e(\mathcal D) = L^2_e(\mathcal D)$ and $H^r_o(\mathcal D) = L^2_o(\mathcal D).$ Note that in the $2D$ case, all notations need to be adapted accordingly by letting $\mathcal M = (0,1)$. When the functional space and the norm are defined in the spatial domain $\mathcal D$, we frequently write $L^2$, $H^r$, $\|\cdot\|_{L^2}$, and $\|\cdot\|_{H^r}$ by omitting $\mathcal D$ when there is no confusion. 

By virtue of \eqref{PE-3} and the boundary condition \eqref{BC-T3}, one can rewrite $w$ as
\begin{equation}\label{w}
    w(\boldsymbol x',z) = -\int_0^z \nabla_h\cdot V(\boldsymbol x', \tilde z) d\tilde z.
\end{equation}
Notice that since $w(z=1)=0$, $V$ satisfies the compatibility condition
\begin{equation}\label{compatibility}
    \int_0^1 \nabla_h\cdot V(\boldsymbol x', \tilde z) d\tilde z =0.
\end{equation}
By Cauchy–Schwarz inequality,
\begin{equation}\label{ine:w}
\begin{split}
     \|w\|^2_{H^r(\mathcal D)} &= \sum\limits_{|\alpha|\leq r} \int_{\mathcal D} \left|D^\alpha \int_0^z \nabla_h\cdot V(\boldsymbol x', \tilde z) d\tilde z \right|^2 d\boldsymbol x 
     \leq C \|\nabla_h V\|^2_{H^r(\mathcal D)}.
\end{split}
\end{equation}

\subsection{Neural Networks}
We will work with the following class of neural networks introduced in \cite{de2022error}.

\begin{definition}\label{def:nn}
Suppose $R \in(0, \infty]$, $L, W \in \mathbb{N}$, and $l_0, \ldots, l_L \in \mathbb{N}$. Let $\sigma: \mathbb{R} \rightarrow \mathbb{R}$ be a twice differentiable activation function, and define
$$
\Theta=\Theta_{L, W, R}:=\bigcup_{L^{\prime} \in \mathbb{N}, L^{\prime} \leq L} \bigcup_{l_0, \ldots, l_L \in\{1, \ldots, W\}} X_{k=1}^{L^{\prime}}\left([-R, R]^{l_k \times l_{k-1}} \times[-R, R]^{l_k}\right) .
$$
For $\theta \in \Theta_{L, W, R}$, we define $\theta_k:=\left(\mathcal{W}_k, b_k\right)$ and $\mathcal{A}_k: \mathbb{R}^{l_{k-1}} \rightarrow \mathbb{R}^{l_k}: x \mapsto \mathcal{W}_k x+b_k$ for $1 \leq k \leq L$ and define $f_k^\theta: \mathbb{R}^{l_{k-1}} \rightarrow \mathbb{R}^{l_k}$ by
$$
f_k^\theta(\eta)= \begin{cases}\mathcal{A}_L^\theta(\eta) & k=L \\ \left(\sigma \circ \mathcal{A}_k^\theta\right)(\eta) & 1 \leq k<L\end{cases}
$$
Denote by $u_\theta: \mathbb{R}^{l_0} \rightarrow \mathbb{R}^{l_L}$ the function that satisfies for all $\eta \in \mathbb{R}^{l_0}$ that
$$
u_\theta(\eta)=\left(f_L^\theta \circ f_{L-1}^\theta \circ \cdots \circ f_1^\theta\right)(\eta).
$$
In our approach to approximating the system \eqref{PE-system}, we assign $l_0= d+1$ and $\eta=(\boldsymbol x, t)$. The neural network that corresponds to parameter $\theta$ and consists of $L$ layers and widths $(l_0,l_1,...,l_L)$ is denoted by $u_\theta$. The first $L-1$ layers are considered hidden layers, where $l_k$ refers to the width of layer $k$, and $\mathcal W_k$ and $b_k$ denote the weights and biases of layer $k$, respectively. The width of $u_\theta$ is defined as the maximum value among ${l_0,\dots,l_L}$.
\end{definition}

\subsection{PINNs Settings} We define the following residuals from the PDE system~\eqref{PE-system}: 

\begin{flalign}\label{residuals-pde}
  \hspace{-0.5cm}\textbf{(PDE residuals)}  \quad \begin{cases}
     \mathcal R_{i,V}[\theta]:= \partial_t V_\theta + V_\theta\cdot\nabla_h V_\theta + w_\theta\partial_z V_\theta +f_0 V_\theta^\perp - \nu_h \Delta_h V_\theta - \nu_z \p_{zz} V_\theta + \nabla_h p_\theta, 
        \\
        \mathcal R_{i,p}[\theta]:= \p_z p_\theta + T_\theta,
        \\
        \mathcal R_{i,T}[\theta] := \partial_t T_\theta + V_\theta\cdot\nabla_h T_\theta + w_\theta\partial_z T_\theta - \kappa_h \Delta_h T_\theta - \kappa_z \p_{zz} T_\theta - Q,
        \\
        \mathcal R_{i,div}[\theta]:= \nabla_h\cdot V_\theta + \p_z w_\theta,
\end{cases}
\end{flalign}
the residuals from the initial conditions~\eqref{IC}:
\begin{flalign}\label{residuals-initial}
  \hspace{-4.8cm}\textbf{(initial residuals)} \qquad\qquad\qquad\qquad\qquad \begin{cases}
         \mathcal R_{t,V}[\theta] := V_\theta(t=0) - V_0,
        \\
        \mathcal R_{t,T}[\theta] := T_\theta(t=0) - T_0,
     \end{cases}
\end{flalign}
and for $s\in \mathbb N$, the residuals from the boundary conditions: 
\begin{equation}\label{residuals-boundary}
    \begin{split}
       \hspace{-0.2cm}\textbf{(boundary residuals) }\mathcal R_{b}[s;\theta] 
       := &\Big\{\sum\limits_{\varphi\in\{V_\theta,w_\theta,p_\theta,T_\theta\}}\sum\limits_{|\alpha|\leq s} \Big[\Big(D^\alpha \varphi(x=1,y,z)- D^\alpha \varphi(x=0,y,z)\Big)^2
        \\
        &+ \Big(D^\alpha \varphi(x,y=1,z)- D^\alpha \varphi(x,y=0,z)\Big)^2 
        \\
        &
        + \Big(D^\alpha \varphi(x,y,z=1)- D^\alpha \varphi(x,y,z=0)\Big)^2 \Big] 
        \\
        &+\sum\limits_{|\alpha|\leq s, \alpha_3=0} \Big[\Big(D^\alpha w_\theta(x,y,z=1)\Big)^2 + \Big(D^\alpha w_\theta(x,y,z=0)\Big)^2\Big] \Big\}^{\frac12},
    \end{split}
\end{equation}
For $s\in \mathbb N$, the generalization error is defined by
\begin{equation}\label{generalization-error}
    \begin{split}
       \hspace{-1.7cm}\textbf{(generalization error) }\qquad\qquad \Eg[s;\theta] := \left(\Eg^i[s;\theta]^2 + \Eg^t[s;\theta]^2 +\Eg^b[s;\theta]^2 + \lambda\Eg^p[s;\theta]^2\right)^{\frac12},
    \end{split}
\end{equation}
where 
\begin{equation*}
    \begin{split}
        &\Eg^i[s;\theta]^2:= \int_0^{\mathcal T}\left(\|\mathcal R_{i,V}[\theta]\|_{H^s(\mathcal D)}^2 + \|\mathcal R_{i,p}[\theta]\|_{H^s(\mathcal D)}^2 + \|\mathcal R_{i,T}[\theta]\|_{H^s(\mathcal D)}^2 + \|\mathcal R_{i,div}[\theta]\|_{H^s(\mathcal D)}^2 \right) dt,
        \\
        &\Eg^t[s;\theta]^2: =\|\mathcal R_{t,V}[\theta]\|_{H^s(\mathcal D)}^2 +  \|\mathcal R_{t,T}[\theta]\|_{H^s(\mathcal D)}^2 ,
        \\
        &\Eg^b[s;\theta]^2 := \int_0^{\mathcal T}\|\mathcal R_{b}[s;\theta]\|_{L^2(\partial \mathcal D)}^2 dt,
        \\
        &\Eg^p[s;\theta]^2 :=  \int_0^{\mathcal T} \left(\|V_\theta\|_{H^{s+3}(\mathcal D)}^2 + \|p_\theta\|_{H^{s+3}(\mathcal D)}^2 + \|w_\theta\|_{H^{s+3}(\mathcal D)}^2 + \|T_\theta\|_{H^{s+3}(\mathcal D)}^2 \right)dt
    \end{split}
\end{equation*}
and the training error is defined by
\begin{equation}\label{training-error}
    \hspace{-0.3cm}\textbf{(training error) } \qquad  \Et[s;\theta;\mathcal S] := \left(\Et^i[s;\theta;\mathcal S_i]^2 + \Et^t[s;\theta;\mathcal S_t]^2 +\Et^b[s;\theta;\mathcal S_b]^2 + \lambda \Et^p[s;\theta;\mathcal S_i]^2\right)^{\frac12},
\end{equation}
where
\begin{equation}\label{training-error-details}
    \begin{split}
        &\Et^i[s;\theta;\mathcal S_i]^2:= \sum\limits_{(t_n,x_n,y_n,z_n)\in \mathcal S_i}\sum\limits_{|\alpha|\leq s}w_n^i\Big[\Big(D^\alpha \mathcal R_{i,V}[\theta](t_n,x_n,y_n,z_n)\Big)^2 + \Big(D^\alpha \mathcal R_{i,p}[\theta](t_n,x_n,y_n,z_n)\Big)^2
        \\
        &\qquad\qquad\qquad\qquad+\Big(D^\alpha \mathcal R_{i,T}[\theta](t_n,x_n,y_n,z_n)\Big)^2  +\Big(D^\alpha \mathcal R_{i,div}[\theta](t_n,x_n,y_n,z_n)\Big)^2  \Big],
        \\
        &\Et^t[s;\theta;\mathcal S_t]^2: =\sum\limits_{(x_n,y_n,z_n)\in \mathcal S_t}\sum\limits_{|\alpha|\leq s}w_n^t\Big[\Big(D^\alpha \mathcal R_{t,V}[\theta](x_n,y_n,z_n)\Big)^2 + \Big(D^\alpha \mathcal R_{t,T}[\theta](x_n,y_n,z_n)\Big)^2 \Big],
        \\
        &\Et^b[s;\theta;\mathcal S_b]^2:  =\sum\limits_{(t_n,x_n,y_n,z_n)\in \mathcal S_b}w_n^b \Big( R_{b}[s;\theta](t_n,x_n,y_n,z_n)\Big)^2,
        \\
        &\Et^p[s;\theta;\mathcal S_i]^2:= \sum\limits_{(t_n,x_n,y_n,z_n)\in \mathcal S_i}\sum\limits_{|\alpha|\leq s+3}w_n^i\Big[\Big(D^\alpha V_\theta(t_n,x_n,y_n,z_n)\Big)^2 + \Big(D^\alpha p_\theta(t_n,x_n,y_n,z_n)\Big)^2
         \\
        &\hspace{2.3in}+\Big(D^\alpha w_\theta(t_n,x_n,y_n,z_n)\Big)^2  +\Big(D^\alpha T_\theta(t_n,x_n,y_n,z_n)\Big)^2  \Big].
    \end{split}
\end{equation}
with quadrature points in space-time constituting data sets $\mathcal S=(\mathcal S_i, \mathcal S_t, \mathcal S_b)$ with $\mathcal S_i \subseteq [0,\mathcal T]\times \mathcal D$, $\mathcal S_t \subseteq \mathcal D$, $\mathcal S_b \in [0,\mathcal T]\times \partial\mathcal D$ and $(w_n^i,w_n^t,w_n^b)$ being the quadrature weights, defined in \eqref{midpoints} and \eqref{weights}, respectively. Here $\Eg^p$ and $\Et^p$ stands the penalty terms.
Finally, for $s\in \mathbb N$, the total error is defined as
\begin{equation}\label{total-error}
    \hspace{-4.1cm}\textbf{(total error) } \qquad \qquad \qquad\mathcal E[s;\theta] := \Big(\int_0^{\mathcal T}\big(\|V-V_\theta\|^2_{H^s} + \|T-T_\theta\|^2_{H^s}\big)dt\Big)^{\frac12}.
\end{equation}
\begin{remark}
\hfill
\begin{enumerate}
    \item In our setting, we impose periodic boundary conditions, $(V,p)$ to be even in $z$, and $(w,T)$ to be odd in $z$. The assumption of evenness and oddness allows us to perform the  periodic extension in the $z$ direction.  This can be ignored when one wants to control the total error from the generalization error. Note also that, the boundary conditions $w|_{z=0,1}=0$ and $D^\alpha w|_{z=0,1}=0$ with $\alpha_3=0$ have physical meanings, and are essential in the error estimate.  Therefore they are included in the boundary residuals $\mathcal R_b[s;\theta]$.
    \item The total error is defined only for $V$ and $T$, as for the primitive equations they are the prognostic variables, while $w$ and $p$ are diagnostic variables that can be recovered from $V$ and $T.$
    \item If the original PINN framework \eqref{eq:originalPINN} were followed, one could first obtain posterior estimates for \ref{SubQ1}--\ref{SubQ2}, meaning that constants would depend on certain norms of the outputs of the neural networks, and then made it a priori by requiring high regularity for the solution. For this approach, see, for instance, \cite[Theorem~3.1 and 3.4, and Corollary~3.14]{de2022error}. We proceeded in an alternative way, inspired by the approach proposed in \cite{biswas2022error}. That is, we consider the additional terms $\Eg^p$ and $\Et^p$ in generalization  and training errors which are able to bound these constants directly by the PDE solution, and therefore achieve an {\it a priori} estimate for the total error. 
\end{enumerate}
\end{remark}

\section{Regularity of Solutions to the Primitive Equations}\label{sec:PE-regularity}
We first give the definition of strong solutions to system \eqref{PE-system} under \ref{Case1}. The following definition is similar to the ones appearing in \cites{cao2007global,cao2016global}.

\begin{definition}\label{def:solution-pe}
Let $\mathcal T>0$ and let $V_0\in H_e^2(\mathcal D)$ and $T_0\in H_o^2(\mathcal D)$. A couple $(V,T)$ is called a strong solution to system
\eqref{PE-system} on $\Omega = [0,\mathcal T]\times \mathcal D$ if

(i) $V$ and $T$ have the regularities
\begin{align*}
&V\in L^\infty(0,\mathcal T; H_e^2(\mathcal D))\cap C([0,\mathcal T];H_e^1(\mathcal D)),\quad && T\in L^\infty(0,\mathcal T; H_o^2(\mathcal D))\cap C([0,\mathcal T];H_o^1(\mathcal D))\\
&(\nabla_h V,\nu_z V_z)\in L^2(0,\mathcal T; H^2(\mathcal D)), \quad &&(\nabla_h T,\kappa_z T_z)\in L^2(0,\mathcal T; H^2(\mathcal D))
\\
&\partial_tV\in L^2(0,\mathcal T; H^1(\mathcal D)), \quad && \partial_tT \in L^2(0,\mathcal T; H^1(\mathcal D)) ;
\end{align*}

(ii) $V$ and $T$ satisfy system \eqref{PE-system} a.e. in $\Omega = [0,\mathcal T]\times \mathcal D$ and the initial condition \eqref{IC}.
\end{definition}

\begin{definition}
A couple $(V,T)$ is called a global strong solution to system \eqref{PE-system} if it is a strong solution on $\Omega=[0,\mathcal T]\times \mathcal D$ for any $\mathcal T>0$.
\end{definition}

The theorem below is from \cites{cao2016global} and concerns the global well-posedness of system \eqref{PE-system} under \ref{Case2}.

\begin{theorem}[{\cite[Theorem~1.3]{cao2016global}}]\label{thm:global-cao}
Suppose that $Q=0$, $V_0\in H_e^2(\mathcal D)$ and $T_0\in H_o^2(\mathcal D)$. Then system \eqref{PE-system} has a unique global strong solution $(V,T)$, which is continuously dependent on the initial data.
\end{theorem}
\begin{remark}
 Theorem \ref{thm:global-cao} works for \ref{Case2}. It can be easily extended to \ref{Case1}, i.e., $\nu_h, \kappa_h, \nu_z,\kappa_z>0$ (see \cite[Proposition~2.6]{cao2016global}). Moreover, under \ref{Case1}, the solution $(V,T)$ satisfies that $V\in L^2(0,\mathcal T; H_e^3(\mathcal D))$ and $T\in L^2(0,\mathcal T; H_o^3(\mathcal D))$ as indicated in Definition \ref{def:solution-pe}. 
Theorem \ref{thm:global-cao} is proved in \cites{cao2016global} with $d=3$, but it also holds when $d=2$. The requirement $Q=0$ can be replaced with $Q$ being regular enough, for example, $Q\in L^\infty(0,\mathcal T; H_o^2(\mathcal D))$ for arbitrary $\mathcal T>0$.
\end{remark}

In order to perform the error analysis for PINNs, we need to establish a higher-order regularity for the solution $(V,T)$, in particular, the continuity in both spatial and temporal variables. This is summarized in the theorem below. 

\begin{theorem}\label{thm:higher-regularity}
Let $k, r\in\mathbb N$, $d\in\{2,3\}$, $r>\frac d2+2k$, $\mathcal T>0$, and denote by $\Omega = [0,\mathcal T]\times \mathcal D.$
Suppose that $V_0\in H_e^r(\mathcal D)$, $T_0\in H_o^r(\mathcal D)$, and $Q\in C^{k-1}([0,\mathcal T]; H_o^r(\mathcal D))$. Then system \eqref{PE-system} has a unique global strong solution $(V,T)$, which depends continuously on the initial data. Moreover, we have
\begin{subequations}\label{regularity-higher}
\begin{align}
    &(V,T)\in L^\infty(0,\mathcal T; H^r(\mathcal D))\cap C([0,\mathcal T];H^{r-1}(\mathcal D)), \label{regularity-higher-sub1}
    \\
    &(\nabla_h V,\nu_z V_z, \nabla_h T,\kappa_z T_z)\in L^2(0,\mathcal T; H^r(\mathcal D))\cap C([0,\mathcal T];H^{r-1}(\mathcal D)),\label{regularity-higher-sub2}
    \\
    &(\partial_tV, \partial_t T)\in L^2(0,\mathcal T; H^{r-1}(\mathcal D)),\label{regularity-higher-sub3}
\end{align}
\end{subequations}
and
\begin{eqnarray}\label{regularity-higher-2}
&(V,T)\in C^k\left(\Omega\right),\quad (w,\nabla p)\in C^{k-1}\left(\Omega\right).
\end{eqnarray}
\end{theorem}

To prove Theorem \ref{thm:higher-regularity}, we shall need the following lemma.
\begin{lemma}[{\cite[Lemma~A.1]{klainerman1981singular}}, see also \cites{beale1984remarks}]\label{lemma:derivative}
 Let $s\geq 1$, and suppose that $f, g\in H^s(\mathcal D)$. Let $\alpha$ be a multi-index such that $|\alpha|\leq s$. Then 
 \begin{equation*}
     \left\|D^\alpha (fg) - f D^\alpha g \right\|_{L^2} \leq C_s\left(\|f\|_{H^s} \|g\|_{L^\infty} + \|\nabla f\|_{L^\infty} \|g\|_{H^{s-1}} \right).
 \end{equation*}
\end{lemma}

\begin{proof}[Proof of Theorem \ref{thm:higher-regularity}]
We perform the proof when $d=3$. The case of $d=2$ follows similarly. Notice that in our setting $r>3$. Let's first consider \ref{Case1}. 

\noindent\textbf{\ref{Case1}: full viscosity.}

We start by showing that, for arbitrary fixed $\mathcal T>0$, we have
\begin{eqnarray*}
(V, T)\in C([0,\mathcal T]; H^3(\mathcal D))\cap L^2(0,\mathcal T; H^4(\mathcal D)), \quad \text{and} \quad  
(\partial_tV, \partial_t T)\in L^2(0,\mathcal T; H^{2}(\mathcal D)).
\end{eqnarray*}
Let $|\alpha|\leq 3$ be a multi-index. Taking $D^\alpha$ derivative on the system \eqref{PE-system}, and then taking the inner product of \eqref{PE-1} with $D^\alpha V$ and \eqref{PE-4} with $D^\alpha T$, by summing over all $|\alpha|\leq 3$.  One has
\begin{equation}\label{est:higher-1}
    \begin{split}
        &\frac12 \frac d{dt} \left(\|V\|_{H^3}^2 + \| T\|_{H^3}^2\right) + \nu_h \|\nabla_h  V\|_{H^3}^2 + \nu_z \|\partial_z  V\|_{H^3}^2
        + \kappa_h \|\nabla_h  T\|_{H^3}^2 + \kappa_z \|\partial_z  T\|_{H^3}^2
        \\
        = & \sum\limits_{|\alpha|\leq 3} \Big(-\left\langle D^\alpha(V\cdot \nabla_h V+wV_z), D^\alpha V  \right\rangle - \left\langle D^\alpha(V\cdot \nabla_h T+wT_z), D^\alpha T  \right\rangle 
        \\
       &\phantom{x}\hspace{9ex} - \left\langle D^\alpha \nabla_h p, D^\alpha V  \right\rangle + \left\langle D^\alpha Q, D^\alpha T  \right\rangle \Big).
    \end{split}
\end{equation}
By integration by parts, thanks to \eqref{BC-T3}, \eqref{PE-2} and \eqref{PE-3}, using the 
Cauchy–Schwarz inequality, Young's inequality and \eqref{ine:w}, one arrives at the following:
\begin{equation}\label{est:higher-2}
\begin{split}
     &\left\langle D^\alpha \nabla_h p, D^\alpha V  \right\rangle = -\left\langle D^\alpha  p, D^\alpha \nabla_h \cdot V  \right\rangle = \left\langle D^\alpha  p, D^\alpha \p_z w  \right\rangle = -\left\langle D^\alpha \p_z p, D^\alpha w  \right\rangle = \left\langle D^\alpha T, D^\alpha w  \right\rangle 
     \\
     &\leq \|T\|_{H^3} \|w\|_{H^3} \leq \|T\|_{H^3} \|\nabla_h V\|_{H^3} \leq \frac14 \nu_h \|\nabla_h V\|_{H^3}^2 + C_{\nu_h} \|T\|_{H^3}^2.
\end{split}
\end{equation}
By the Cauchy–Schwarz inequality and Young's inequality, one deduces
\begin{equation}\label{est:higher-3}
    \left\langle D^\alpha Q, D^\alpha T  \right\rangle \leq \|Q\|_{H^3} \|T\|_{H^3} \leq \frac12 \|Q\|_{H^3}^2 + \frac12 \|T\|_{H^3}^2.
\end{equation}
Using Lemma \ref{lemma:derivative}, integration by parts, and the boundary condition and \eqref{PE-3}, from the Cauchy-Schwarz inequality, Young's inequality and the Sobolev inequality, for all $|\alpha|\leq 3$ one has
\begin{equation}\label{est:higher-4}
\begin{split}
    &\left\langle D^\alpha(V\cdot \nabla_h V+wV_z), D^\alpha V  \right\rangle + \left\langle D^\alpha(V\cdot \nabla_h T+wT_z), D^\alpha T \right\rangle 
    \\
    =& \left\langle D^\alpha(V\cdot \nabla_h V) - V\cdot D^\alpha\nabla_h V + D^\alpha(wV_z) - w D^\alpha V_z, D^\alpha V  \right\rangle
    \\
    &+ \left\langle D^\alpha(V\cdot \nabla_h T) - V\cdot D^\alpha\nabla_h T + D^\alpha(wT_z) - w D^\alpha T_z, D^\alpha T  \right\rangle
    \\
    &+ \underbrace{\left\langle  V\cdot D^\alpha\nabla_h V +  w D^\alpha V_z, D^\alpha V  \right\rangle}_{=0} + \underbrace{\left\langle  V\cdot D^\alpha\nabla_h T +  w D^\alpha T_z, D^\alpha T  \right\rangle}_{=0}
    \\
    \leq & C\Big(\|V\|_{H^3} \|\nabla_h V\|_{L^\infty} + \|\nabla V\|_{L^\infty} \|\nabla_h V\|_{H^{2}} + \|w\|_{H^3} \|V_z\|_{L^\infty} + \|\nabla w\|_{L^\infty} \|V_z\|_{H^{2}}\Big)\|V\|_{H^3} 
    \\
    &+ C\Big(\|V\|_{H^3} \|\nabla_h T\|_{L^\infty} + \|\nabla V\|_{L^\infty} \|\nabla_h T\|_{H^{2}} + \|w\|_{H^3} \|T_z\|_{L^\infty} + \|\nabla w\|_{L^\infty} \|T_z\|_{H^{2}}\Big) \|T\|_{H^3} 
    \\
    \leq & C \|V\|_{H^3}^2 \|\nabla_h V\|_{H^3} + C \|V\|_{H^3} \|T\|_{H^3}^2 + C \|T\|_{H^3}^2 \|\nabla_h V\|_{H^3}
    \\
    \leq & \frac14 \nu_h \|\nabla_h V\|_{H^3}^2  + C_{\nu_h} (1+\|V\|_{H^3}^2+ \|T\|_{H^3}^2 ) (\|V\|_{H^3}^2+ \|T\|_{H^3}^2 ).
\end{split}
\end{equation}
Note that we have applied Lemma \ref{lemma:derivative} for the first inequality.
Combine the estimates \eqref{est:higher-1}--\eqref{est:higher-4}, we obtain
\begin{equation*}
    \begin{split}
        & \frac d{dt} \left(\|V\|_{H^3}^2 + \| T\|_{H^3}^2\right) + \nu_h \|\nabla_h  V\|_{H^3}^2 + \nu_z \|\partial_z  V\|_{H^3}^2
        + \kappa_h \|\nabla_h  T\|_{H^3}^2 + \kappa_z \|\partial_z  T\|_{H^3}^2
        \\
        \leq & C_{\nu_h} (1+\|V\|_{H^3}^2+ \|T\|_{H^3}^2 )(\|V\|_{H^3}^2+ \|T\|_{H^3}^2 ) + C \|Q\|_{H^3}^2.
    \end{split}
\end{equation*}
From Theorem \ref{thm:global-cao} we know that $V,T\in L^2(0,\mathcal T; H^3(\mathcal D))$
for arbitrary $\mathcal T>0$. By Gronwall inequality, for any $t\in[0,\mathcal T]$,
\begin{equation*}
    \begin{split}
    &\|V(t)\|_{H^3}^2 + \| T(t)\|_{H^3}^2 + \int_0^t \left(\nu_h \|\nabla_h  V(\tilde t)\|_{H^3}^2 + \nu_z \|\partial_z  V(\tilde t)\|_{H^3}^2
        + \kappa_h \|\nabla_h  T(\tilde t)\|_{H^3}^2 + \kappa_z \|\partial_z  T(\tilde t)\|_{H^3}^2 \right) d\tilde t
        \\
        &\leq \left(\|V_0\|_{H^3}^2 + \| T_0\|_{H^3}^2 + C \int_0^{\mathcal T} \|Q(\tilde t)\|_{H^3}^2 d\tilde t \right) \exp\left(\int_0^{\mathcal T} C_{\nu_h} (1+\|V(\tilde t)\|_{H^3}^2+ \|T(\tilde t)\|_{H^3}^2 ) d\tilde t \right) <\infty. 
    \end{split}
\end{equation*}
Therefore, we get
\begin{eqnarray*}
&(V, T)\in L^\infty(0,\mathcal T; H^3(\mathcal D))\cap L^2(0,\mathcal T; H^4(\mathcal D)).
\end{eqnarray*}
Now for any $|\alpha|\leq 2$, taking $D^\alpha$ derivative on system \eqref{PE-system} and then taking the inner product of \eqref{PE-1} and \eqref{PE-2} with $\varphi\in \{f\in L^2(\mathcal D): \nabla\cdot f =0\}$, one has
\begin{equation*}
\begin{split}
     \left\langle D^\alpha\partial_t V, \varphi \right\rangle = & -\left\langle D^\alpha(V\cdot \nabla_h V), \varphi \right\rangle -\left\langle D^\alpha(w \partial_z V), \varphi \right\rangle + \nu_h \left\langle D^\alpha\Delta_h V, \varphi \right\rangle + \nu_z \left\langle D^\alpha\partial_{zz} V, \varphi \right\rangle - f_0 \left\langle D^\alpha V^\perp, \varphi \right\rangle
     \\
     &- \left\langle D^\alpha\nabla_h p, \varphi \right\rangle - \left\langle D^\alpha p_z, \varphi \right\rangle - \left\langle D^\alpha T, \varphi \right\rangle .
\end{split}
\end{equation*}
By the Cauchy-Schwarz inequality, recalling that $H^s$ is a Banach algebra when $s>\frac d2$, we have
\begin{equation*}
   |\left\langle D^\alpha\partial_t V, \varphi \right\rangle| \leq \left(C \|V\|_{H^3}^2 + C_{\nu_h,\nu_z,f_0} \|V\|_{H^4} + \|T\|_{H^2} \right) \|\varphi\|_{L^2},
\end{equation*}
where we have consecutively used integration by parts and $\nabla\cdot \varphi=0$ to get 
$
    \left\langle D^\alpha\nabla_h p, \varphi \right\rangle + \left\langle D^\alpha p_z, \varphi \right\rangle =0.
$
Since the inequality above is true for any $|\alpha|\leq 2$, and the space $\{f\in L^2(\mathcal D): \nabla\cdot f =0\}$ is dense in $L^2(\mathcal D)$, from the regularity of $V$ and $T$, one deduces $$\partial_t V\in L^2(0,\mathcal T; H^2(\mathcal D)).$$ A similar argument yields $$\partial_t T\in L^2(0,\mathcal T; H^2(\mathcal D)).$$ Applying the Lions-Magenes theorem (see \textit{e.g. }\cite[Chapter 3, Lemma 1.2]{temam2001navier}), together with the regularity of $V$, $T$, $\partial_t V$, $\partial_t T$, we obtain
\begin{equation*}
   (V, T)\in C([0,\mathcal T];H^3(\mathcal D)).
\end{equation*}
This completes the proof of \eqref{regularity-higher} with $r = 3$. The proof of $r = 4$ and all subsequent $r$ is then obtained by repeating the same argument. Therefore under \ref{Case1}, we achieve \eqref{regularity-higher} and moreover, $(V, T)\in C([0,\mathcal T];H^r(\mathcal D))$. 

Next, we show \eqref{regularity-higher-2}. Taking the horizontal divergence on equation \eqref{PE-1}, integrating with respect to $z$ from $0$ to $1$, and taking the vertical derivative on equation \eqref{PE-2} gives
\begin{equation}\label{pressure}
\begin{split}
    \Delta p(\boldsymbol x) =& -\int_0^1 \nabla_h \cdot \left( V\cdot \nabla_h  V + w\partial_z V -\nu_h \Delta_h V - \nu_z \partial_{zz} V +f_0 V^\perp \right)(\boldsymbol x',z) dz - T_z(\boldsymbol x)
    \\
    =&-\int_0^1 \nabla_h \cdot \left( V\cdot \nabla_h  V + w\partial_z V +f_0 V^\perp \right)(\boldsymbol x',z) dz - T_z(\boldsymbol x),
\end{split}
\end{equation}
where the viscosity terms disappear due to \eqref{BC-T3} and \eqref{compatibility}.
We first consider $k=1$. Since $r> \frac d2+2k$, we know $H^{r-2}$ is a Banach algebra.  Since $(V,T)\in C([0,\mathcal T];H^r(\mathcal D))$, one has
$
\Delta p \in C([0,\mathcal T];H^{r-2}(\mathcal D))$, and thus $\nabla p \in C([0,\mathcal T];H^{r-1}(\mathcal D))$.
This implies that $\partial_t V\in C([0,\mathcal T];H^{r-2}(\mathcal D))$ and therefore, $V\in C^1([0,\mathcal T];H^{r-2}(\mathcal D))$ and $w\in C^1([0,\mathcal T];H^{r-3}(\mathcal D))$. Moreover, since $Q\in C^{k-1}([0,\mathcal T]; H^r(\mathcal D))$, one has $\partial_t T\in C([0,\mathcal T];H^{r-2}(\mathcal D))$, consequently $T\in C^1([0,\mathcal T];H^{r-2}(\mathcal D))$.

When $k= 2$, since $H^{r-4}$ is a Banach algebra, we can take the time derivative on equation~\eqref{pressure} and get that $\Delta p_t \in C([0,\mathcal T];H^{r-4}(\mathcal D))$, and therefore $\nabla p \in C^1([0,\mathcal T];H^{r-3}(\mathcal D))$. This implies $\partial_t V \in C^1([0,\mathcal T];H^{r-4}(\mathcal D))$ and therefore $V\in C^2([0,\mathcal T];H^{r-4}(\mathcal D))$. One can also get $w\in C^2([0,\mathcal T];H^{r-5}(\mathcal D))$ and $T\in C^2([0,\mathcal T];H^{r-4}(\mathcal D)).$ By repeating the above procedure, one will obtain
\begin{equation*}
\begin{split}
    &(V, T)\in \cap_{l=0}^k C^l([0,\mathcal T];H^{r-2l}(\mathcal D)), 
    \\
    &w \in \cap_{l=0}^k C^l([0,\mathcal T];H^{r-2l-1}(\mathcal D)), 
    \\
    &\nabla p \in \cap_{l=0}^{k-1} C^l([0,\mathcal T];H^{r-2l-1}(\mathcal D)).
\end{split}
\end{equation*}
Then by the Sobolev embedding theorem and  $r> \frac d2 + 2k$, we know that $H^{r-2l}(\mathcal D) \subset C^{2k-2l}(\mathcal D)$ for $0\leq l \leq k$ and $H^{r-2l-1}(\mathcal D) \subset C^{2k-2l-1}(\mathcal D)$ for $0\leq l \leq k-1$. Therefore,
\begin{equation*}
    (V,T)\in C^k\left(\Omega\right),\quad (w,\nabla p)\in C^{k-1}\left(\Omega\right).
\end{equation*}

\noindent \textbf{\ref{Case2}: only horizontal viscosity.}

Under \ref{Case2}, the proof of \eqref{regularity-higher} when $r=3$ is more technically involved. The key difference is in the estimate of the nonlinear term. 

When $D^\alpha = \partial_{z}^3$,   integration by parts yields
\begin{equation}\label{est:horizontal-1}
    \begin{split}
    &\left\langle \partial_{z}^3(V\cdot \nabla_h V+wV_z), \partial_{z}^3 V  \right\rangle + \left\langle \partial_{z}^3(V\cdot \nabla_h T+wT_z), \partial_{z}^3 T \right\rangle 
    \\
    =& \left\langle \partial_{z}^3(V\cdot \nabla_h V) - V\cdot \partial_{z}^3\nabla_h V + \partial_{z}^3(wV_z) - w \partial_{z}^3 V_z, \partial_{z}^3 V  \right\rangle
    \\
    &+ \left\langle \partial_{z}^3(V\cdot \nabla_h T) - V\cdot \partial_{z}^3\nabla_h T + \partial_{z}^3(wT_z) - w \partial_{z}^3 T_z, \partial_{z}^3 T  \right\rangle
    \\
    &+ \underbrace{\left\langle  V\cdot \partial_{z}^3\nabla_h V +  w \partial_{z}^3 V_z, \partial_{z}^3 V  \right\rangle}_{=0} + \underbrace{\left\langle  V\cdot \partial_{z}^3\nabla_h T +  w \partial_{z}^3 T_z, \partial_{z}^3 T  \right\rangle}_{=0}
    \\
    = & \left\langle \partial_{z}^3(V\cdot \nabla_h V) - V\cdot \partial_{z}^3\nabla_h V , \partial_{z}^3 V  \right\rangle + \left\langle \partial_{z}^3(V\cdot \nabla_h T) - V\cdot \partial_{z}^3\nabla_h T , \partial_{z}^3 T  \right\rangle
    \\
    &+ \left\langle  \partial_{z}^3(wV_z) - w \partial_{z}^3 V_z, \partial_{z}^3 V  \right\rangle + \left\langle  \partial_{z}^3(wT_z) - w \partial_{z}^3 T_z, \partial_{z}^3 T  \right\rangle
    \\
    : =& I_1 + I_2 + I_3 + I_4.
\end{split}
\end{equation}
Using Lemma \ref{lemma:derivative}, together with the Cauchy-Schwarz inequality, Young's inequality, and the Sobolev inequality, one obtains
\begin{equation}\label{est:horizontal-2}
    \begin{split}
        I_1 + I_2 &= \left\langle \partial_{z}^3(V\cdot \nabla_h V) - V\cdot \partial_{z}^3\nabla_h V , \partial_{z}^3 V  \right\rangle + \left\langle \partial_{z}^3(V\cdot \nabla_h T) - V\cdot \partial_{z}^3\nabla_h T , \partial_{z}^3 T  \right\rangle
        \\
        & \leq C\Big(\|V\|_{H^3} \|\nabla_h V\|_{L^\infty} + \|\nabla V\|_{L^\infty} \|\nabla_h V\|_{H^{2}} \Big)\|V\|_{H^3} 
    \\
    & \qquad + C\Big(\|V\|_{H^3} \|\nabla_h T\|_{L^\infty} + \|\nabla V\|_{L^\infty} \|\nabla_h T\|_{H^{2}} \Big) \|T\|_{H^3} 
    \\
    & \leq C \left(\|\nabla_h V\|_{H^{2}}  +  \|\nabla_h T\|_{H^{2}}\right)\left( \|V\|_{H^3}^2 + \|T\|_{H^3}^2\right).
    \end{split}
\end{equation}
For the estimates of $I_3$, we use  \eqref{PE-3},  Young's inequality, the H\"older inequality and the Sobolev inequality, to achieve
\begin{equation}\label{est:horizontal-3}
\begin{split}
    I_3 = &\left\langle  \partial_{z}^3(wV_z) - w \partial_{z}^3 V_z, \partial_{z}^3 V  \right\rangle
    \\
    \leq & C \left|\left\langle \underbrace{(\nabla_h \cdot V)}_{L^\infty} \underbrace{\partial_{z}^3 V}_{L^2}, \underbrace{\partial_{z}^3 V}_{L^2} \right\rangle \right| +  C  \left|\left\langle \underbrace{(\nabla_h \cdot V_z)}_{L^6} \underbrace{\partial_{z}^2 V}_{L^3}, \underbrace{\partial_{z}^3 V}_{L^2} \right\rangle \right|+ C \left|\left\langle \underbrace{(\nabla_h \cdot \partial_{z}^2 V)}_{L^2} \underbrace{V_z}_{L^\infty}, \underbrace{\partial_{z}^3 V}_{L^2}  \right\rangle \right|
    \\
    \leq & C\|\nabla_h V\|_{H^2} \|V\|_{H^3}^2.
\end{split}
\end{equation}
Similarly, one can deduce that
\begin{align}
    I_4 = &\left\langle  \partial_{z}^3(wT_z) - w \partial_{z}^3 T_z, \partial_{z}^3 T  \right\rangle
    \leq C\|\nabla_h V\|_{H^2} \|T\|_{H^3}^2. \label{est:horizontal-4}
\end{align}
Combining \eqref{est:horizontal-1}--\eqref{est:horizontal-4} yields
\begin{multline}\label{est:horizontal-5}
     \left\langle \partial_{z}^3(V\cdot \nabla_h V+wV_z), \partial_{z}^3 V  \right\rangle + \left\langle \partial_{z}^3(V\cdot \nabla_h T+wT_z), \partial_{z}^3 T \right\rangle  \\
     \leq C \left(\|\nabla_h V\|_{H^{2}}  +  \|\nabla_h T\|_{H^{2}}\right)\left( \|V\|_{H^3}^2 + \|T\|_{H^3}^2\right).
\end{multline}

When $|\alpha|\leq 3$ and $\alpha\neq (0,0,3)$, one has 
$
    \|D^\alpha V\|_{L^2} \leq \|\nabla_h V\|_{H^2} + \|V\|_{H^2}, \|D^\alpha T\|_{L^2} \leq \|\nabla_h T\|_{H^2} + \|T\|_{H^2}.
$
Therefore, an estimate to replace \eqref{est:higher-4} for \ref{Case2} is
\begin{equation}\label{est:horizontal-6}
\begin{split}
    &\left\langle D^\alpha(V\cdot \nabla_h V+wV_z), D^\alpha V  \right\rangle + \left\langle D^\alpha(V\cdot \nabla_h T+wT_z), D^\alpha T \right\rangle 
    \\
    \leq & C\Big(\|V\|_{H^3} \|\nabla_h V\|_{L^\infty} + \|\nabla V\|_{L^\infty} \|\nabla_h V\|_{H^{2}} + \|w\|_{H^3} \|V_z\|_{L^\infty} + \|\nabla w\|_{L^\infty} \|V_z\|_{H^{2}}\Big)(\|\nabla_h V\|_{H^2} + \|V\|_{H^2})
    \\
    &+ C\Big(\|V\|_{H^3} \|\nabla_h T\|_{L^\infty} + \|\nabla V\|_{L^\infty} \|\nabla_h T\|_{H^{2}} + \|w\|_{H^3} \|T_z\|_{L^\infty} + \|\nabla w\|_{L^\infty} \|T_z\|_{H^{2}}\Big) (\|\nabla_h T\|_{H^2} + \|T\|_{H^2})
    \\
    \leq & C \|V\|_{H^3}^2 \|\nabla_h V\|_{H^2} + C\|\nabla_h V\|_{H^3}\|V\|_{H^3}(\|\nabla_h V\|_{H^2} + \|V\|_{H^2}) 
    \\
    &\qquad\qquad+ C \|V\|_{H^3} \|T\|_{H^3} \|\nabla_h T\|_{H^2} + C\|\nabla_h V\|_{H^3}\|T\|_{H^3}(\|\nabla_h T\|_{H^2} + \|T\|_{H^2}) 
    \\
    \leq & \frac14 \nu_h \|\nabla_h V\|_{H^3}^2 + C_{\nu_h} (1+\| V\|_{H^2}^2+ \| T\|_{H^2}^2+\|\nabla_h V\|_{H^2}^2+ \|\nabla_h T\|_{H^2}^2 ) (\|V\|_{H^3}^2+ \|T\|_{H^3}^2 ).
\end{split}
\end{equation}
Combining \eqref{est:horizontal-5}, \eqref{est:horizontal-6}, \eqref{est:higher-1}--\eqref{est:higher-3} yields
\begin{equation*}
    \begin{split}
        & \frac d{dt} \left(\|V\|_{H^3}^2 + \| T\|_{H^3}^2\right) + \nu_h \|\nabla_h  V\|_{H^3}^2 
        + \kappa_h \|\nabla_h  T\|_{H^3}^2 
        \\
        \leq & C_{\nu_h} (1+\| V\|_{H^2}^2+ \| T\|_{H^2}^2+\|\nabla_h V\|_{H^2}^2+ \|\nabla_h T\|_{H^2}^2 )(\|V\|_{H^3}^2+ \|T\|_{H^3}^2 ) + C \|Q\|_{H^3}^2.
    \end{split}
\end{equation*}
From Theorem \ref{thm:global-cao}, we know that
$(V,T)\in L^\infty(0,\mathcal T; H^2(\mathcal D)), (\nabla_h V, \nabla_h T) \in L^2(0,\mathcal T; H^2(\mathcal D))$
for arbitrary $\mathcal T>0$. Thanks to Gronwall inequality, for any $t\in[0,\mathcal T]$,
\begin{align*}
     &\|V(t)\|_{H^3}^2 + \| T(t)\|_{H^3}^2 + \int_0^t \left(\nu_h \|\nabla_h  V(\tilde t)\|_{H^3}^2 
        + \kappa_h \|\nabla_h  T(\tilde t)\|_{H^3}^2  \right) d\tilde t
        \\
        &\leq \left(\|V_0\|_{H^3}^2 + \| T_0\|_{H^3}^2 + C \int_0^{\mathcal T} \|Q(\tilde t)\|_{H^3}^2 d\tilde t \right) 
        \\
        &\qquad  \times \exp\left(\int_0^{\mathcal T} C_{\nu_h} (1+\|V(\tilde t)\|_{H^2}^2+ \|T(\tilde t)\|_{H^2}^2 + \|\nabla_h V(\tilde t)\|_{H^2}^2+ \|\nabla_h T(\tilde t)\|_{H^2}^2) d\tilde t \right) <\infty. 
\end{align*}
Therefore, we obtain
\begin{eqnarray*}
&(V, T)\in L^\infty(0,\mathcal T; H^3(\mathcal D)), \qquad (\nabla_h V, \nabla_h T) \in L^2(0,\mathcal T; H^3(\mathcal D)).
\end{eqnarray*}

Now following a similar argument as in \ref{Case1}, and using the Lions-Magenes theorem, we obtain.
\begin{equation*}
    (V,T, \nabla_hV, \nabla_h T)\in  C([0,\mathcal T];H^{2}(\mathcal D)), \qquad (\partial_t V, \partial_t T)\in L^2(0,\mathcal T; H^{2}(\mathcal D)).
\end{equation*}
Notice that we cannot achieve $(V,T) \in  C([0,\mathcal T];H^{3}(\mathcal D))$ since $(V,T) \not\in L^2(0,\mathcal T; H^{4}(\mathcal D))$. The proof of \eqref{regularity-higher} for $r=4$ follows the same argument as in \ref{Case1}. By repeating this procedure one can eventually have \eqref{regularity-higher} for \ref{Case2}.

Finally, as $(V,\nabla_h V, T, \nabla_h T)\in C([0,\mathcal T];H^{r-1}(\mathcal D))$, one can repeat the argument as in \ref{Case1} and gets
\begin{equation*}
\begin{split}
    &(V, T)\in  C([0,\mathcal T];H^{r-1}(\mathcal D)) \cap_{l=1}^k C^l([0,\mathcal T];H^{r-2l}(\mathcal D)), 
    \\
    &w \in \cap_{l=0}^k C^l([0,\mathcal T];H^{r-2l-1}(\mathcal D)), 
    \\
    &\nabla p \in \cap_{l=0}^{k-1} C^l([0,\mathcal T];H^{r-2l-1}(\mathcal D)).
\end{split}
\end{equation*}
This concludes the proof of \eqref{regularity-higher-2}.
\end{proof}

\section{Error Estimates for PINNs}\label{sec:error-estimate}
\subsection{Generalization Error Estimates}
In this section, we answer the question \ref{SubQ1} raised in the introduction: given $\varepsilon > 0$, does there exist a neural network $(V_\theta,w_\theta,p_\theta,T_\theta)$ such that the corresponding generalization error $\Eg[s;\theta]$ defined in \eqref{generalization-error} satisfies $\Eg[s;\theta]< \varepsilon$? We have the following result. 

\begin{theorem}[Answer of \ref{SubQ1}]\label{thm:generalization-error}
Let $d\in\{2,3\}$, $n\geq 2$, $k\geq 5$, $r> \frac d2 + 2k$, and $\mathcal T>0$. Suppose that $V_0\in H_e^r(\mathcal D)$, $T_0\in H_o^r(\mathcal D)$, and $Q\in C^{k-1}([0,\mathcal T]; H_o^r(\mathcal D))$. Then for any given $\varepsilon>0$ and for any $0\leq \ell\leq k-5$, there exist $\lambda = \mathcal O(\varepsilon^2)$ small enough and $N>5$ large enough depending on $\varepsilon$ and $\ell$, and tanh neural networks $\widehat V_i:= ( V_i^N)_\theta$, $\widehat w:= w^N_\theta$, $\widehat p:= p^N_\theta$, and $\widehat T:= T^N_\theta$, with $i=1,...,d-1$, each with two hidden layers, of widths at most $3\left\lceil\frac{k+n-2}{2}\right\rceil\abs{P_{k-1,d+2}}+d(N-1) + \lceil T(N-1) \rceil$ and $3\left\lceil\frac{d+1+n}{2}\right\rceil\abs{P_{d+2,d+2}}N^{d}\lceil TN \rceil$, such that the generalization error satisfies 
\[
\Eg[\ell;\theta] \leq \varepsilon.
\]
\end{theorem}

\begin{proof}
For simplicity, we treat the case of $d=3$. The proof of $d=2$ follows in a similar way.

From Theorem \ref{thm:higher-regularity}, we know that system \eqref{PE-system} has a unique global strong solution $(V,T)$ and
\begin{eqnarray*}
&(V,T)\in C^k\left(\Omega\right) \subset H^{k}\left(\Omega\right),
\\
&(w,p,\nabla p)\in C^{k-1}\left(\Omega\right) \subset H^{k-1}\left(\Omega\right).
\end{eqnarray*}
By applying Lemma \ref{lemma:appro}, for fixed $N>5$, there exist tanh neural networks $\widehat  V_1$, $\widehat  V_2$, $\widehat w$, $\widehat p$, and $\widehat T$, each with two hidden layers, of widths at most $3\left\lceil\frac{k+n-2}{2}\right\rceil\abs{P_{k-1,5}}+4(N-1) + \lceil T(N-1)\rceil$ and $3 \left\lceil\frac{4+n}{2}\right\rceil\abs{P_{5,5}} N^{3} \lceil TN\rceil$ such that for every $1\leq s \leq k-1$, one has
\begin{align}\label{pre-error}
     &\|V_i - \widehat V_i\|_{H^{s}(\Omega)}  \leq C_{V} \frac{1+\ln^{s}N}{N^{k-s}},\qquad \|w - \widehat w\|_{H^{s-1}(\Omega)} \leq C_{w} \frac{1+\ln^{s-1}N}{N^{k-s}},\nonumber
    \\
    &\|p - \widehat  p\|_{H^{s-1}(\Omega)} \leq C_{p} \frac{1+\ln^{s-1}N}{N^{k-s}},\qquad\|T - \widehat T\|_{H^{s}(\Omega)} \leq  C_{T} \frac{1+\ln^{s}N}{N^{k-s}},
\end{align}
where the constants $C_{V}, C_{w}, C_{p}$ and $C_{T}$ are defined according to Lemma \ref{lemma:appro}. Denote by $\widehat V = (\widehat V_1,\widehat V_2)$ and $(\partial_1, \partial_2)=(\partial_x, \partial_y)$. For $0\leq \ell \leq k-5$ and $i=1,2$, one has
\begin{align}
     &\|\partial_t  V_i - \partial_t \widehat V_i\|_{H^\ell(\Omega)} \leq \|V_i - \widehat V_i\|_{H^{\ell+1}(\Omega)} \leq C_{V} \frac{1+\ln^{\ell+1}N}{N^{k-\ell -1}}, \nonumber
     \\
     &\|\partial_i p - \partial_i \widehat  p\|_{H^{\ell}(\Omega)}, \, \|\partial_z p - \partial_z \widehat  p\|_{H^{\ell}(\Omega)} \leq \|p - \widehat  p\|_{H^{\ell+1}(\Omega)} \leq C_{p} \frac{1+\ln^{\ell+1}N}{N^{k-\ell-2}}, \nonumber
     \\
     &\| \Delta_h V_i - \Delta_h \widehat V_i\|_{H^\ell(\Omega)}, \, \| \partial_{zz} V_i - \partial_{zz} \widehat V_i\|_{H^\ell(\Omega)}  \leq \| V_i -  \widehat V_i\|_{H^{\ell+2}(\Omega)} \leq C_{V} \frac{1+\ln^{\ell+2}N}{N^{k-\ell -2}}, \nonumber
     \\
     & \|\nabla_h \cdot V+ \partial_z w - (\nabla_h \cdot \widehat V + \partial_z \widehat w)\|_{H^{\ell}(\Omega)} \leq \sum\limits_{i=1}^2\|V_i - \widehat V_i\|_{H^{\ell+1}(\Omega)} + \|w - \widehat w\|_{H^{\ell+1}(\Omega)} \nonumber
     \\
     &\hspace{170pt}\leq (C_{V} +C_w) \frac{1+\ln^{\ell+1}N}{N^{k-\ell -2}}, \nonumber
     \\
     &\|\partial_t  T - \partial_t \widehat T\|_{H^\ell(\Omega)} \leq \|T - \widehat T\|_{H^{\ell+1}(\Omega)} \leq C_{T} \frac{1+\ln^{\ell+1}N}{N^{k-\ell -1}},\nonumber
     \\
     &\| \Delta_h T - \Delta_h \widehat T\|_{H^\ell(\Omega)}, \, \| \partial_{zz} T - \partial_{zz} \widehat T\|_{H^\ell(\Omega)}  \leq \| T -  \widehat T\|_{H^{\ell+2}(\Omega)} \leq C_{T} \frac{1+\ln^{\ell+2}N}{N^{k-\ell -2}}. \label{appro-est-1}
\end{align}
For the nonlinear terms we use the Sobolev inequality and \eqref{pre-error}, since $0\leq \ell\leq k-5$,
\begin{equation}\label{appro-est-2}
    \begin{split}
        &\| V\cdot \nabla_h V_i - \widehat V \cdot \nabla_h \widehat V_i\|_{H^\ell(\Omega)} \leq \| (V- \widehat V)\cdot \nabla_h V_i\|_{H^\ell(\Omega)} + \|  \widehat V \cdot \nabla_h (V_i-\widehat V_i)\|_{H^\ell(\Omega)} 
        \\
        \leq & C \|V_i\|_{W^{k-4,\infty}(\Omega)} \max_i\|V_i - \widehat V_i\|_{H^\ell(\Omega)} + C \|\widehat V\|_{W^{k-5,\infty}(\Omega)}  \|V_i - \widehat V_i\|_{H^{\ell+1}(\Omega)}
        \\
        \leq & C\left(\|V\|_{W^{k-4,\infty}(\Omega)}+ \|\widehat V\|_{W^{k-5,\infty}(\Omega)}\right) C_{V} \frac{1+\ln^{\ell+1}N}{N^{k-\ell -1}}
        \\
        \leq & C(\|V\|_{H^{k-2}(\Omega)} + C_V \frac{1+\ln^{k-3}N}{N^3})C_V\frac{1+\ln^{\ell+1}N}{N^{k-\ell -1}},
        \\
        &\| w\partial_z V_i - \widehat w \partial_z \widehat V_i\|_{H^\ell(\Omega)} \leq \| (w - \widehat w)\partial_z V_i\|_{H^\ell(\Omega)} + \|  \widehat w \partial_z (V_i-\widehat V_i)\|_{H^\ell(\Omega)} 
        \\
        \leq & C \|V_i\|_{W^{k-4,\infty}(\Omega)} \|w - \widehat w\|_{H^\ell(\Omega)} + C \|\widehat w\|_{W^{k-5,\infty}(\Omega)}  \|V_i - \widehat V_i\|_{H^{\ell+1}(\Omega)}
        \\
        \leq & C\left(\|V\|_{W^{k-4,\infty}(\Omega)}+\|\hat w\|_{W^{k-5,\infty}(\Omega)} \right) (C_w + C_{V})\frac{1+\ln^{\ell+1}N}{N^{k-\ell -1}}
        \\
        \leq &C\left(\|V\|_{H^{k-2}(\Omega)}+\|w\|_{H^{k-3}(\Omega)} + C_w\frac{1+\ln^{k-3}N}{N^{2}}  \right) (C_w + C_{V})\frac{1+\ln^{\ell+1}N}{N^{k-\ell -1}} .
    \end{split}
\end{equation}
Here we have used \eqref{pre-error} to obtain
\begin{align*}
  &C\|\widehat V\|_{W^{k-5,\infty}(\Omega)} \leq C\|\widehat V\|_{H^{k-3}(\Omega)} \leq C(\|V - \widehat V\|_{H^{k-3}(\Omega)} + \|V\|_{H^{k-3}(\Omega)}) \leq C(\|V\|_{H^{k-2}(\Omega)} + C_{V} \frac{1+\ln^{k-3} N}{N^3}),
  \\
  &C\|\widehat w\|_{W^{k-5,\infty}(\Omega)} \leq C\|\widehat w\|_{H^{k-3}(\Omega)}\leq C(\|w - \widehat w\|_{H^{k-3}(\Omega)} + \|w\|_{H^{k-3}(\Omega)}) \leq C(\|w\|_{H^{k-3}(\Omega)} + C_{w} \frac{1+\ln^{k-3} N}{N^2}).
\end{align*}
A similar calculation of \eqref{appro-est-2} yields
\begin{equation}\label{appro-est-3}
    \begin{split}
       & \| V\cdot \nabla_h T - \widehat V \cdot \nabla_h \widehat T\|_{H^\ell(\Omega)} 
       \\
       \leq  &C\left(\|T\|_{H^{k-2}(\Omega)} + \|V\|_{H^{k-2}(\Omega)}+ C_V \frac{1+\ln^{k-3}N}{N^3}\right)(C_V+C_T)\frac{1+\ln^{\ell+1}N}{N^{k-\ell -1}},
        \\
        &\| w\partial_z T - \widehat w \partial_z \widehat T\|_{H^\ell(\Omega)} 
        \\
        \leq & C\left(\|T\|_{H^{k-2}(\Omega)} + \|w\|_{H^{k-3}(\Omega)}+ C_w \frac{1+\ln^{k-3} N}{N^2}\right) (C_w + C_{T})\frac{1+\ln^{\ell+1}N}{N^{k-\ell -1}}.
    \end{split}
\end{equation}
Combining \eqref{appro-est-1}--\eqref{appro-est-3} gives
\begin{align}\label{appro-est-pde-1}
     &\Big\|\partial_t \widehat V + \widehat V \cdot\nabla_h \widehat V + \widehat w \partial_z \widehat V + \nabla_h \widehat p + f_0 \widehat V^\perp - \nu_h \Delta_h \widehat V - \nu_z \partial_{zz} \widehat V\Big\|_{H^\ell(\Omega)} \nonumber
    \\
    &=\Big\|\left(\partial_t \widehat V + \widehat V \cdot\nabla_h \widehat V + \widehat w \partial_z \widehat V + \nabla_h \widehat p + f_0 \widehat V^\perp - \nu_h \Delta_h \widehat V - \nu_z \partial_{zz} \widehat V\right) \nonumber
    \\
    &\qquad - \underbrace{\left(\partial_t  V+  V \cdot\nabla_h  V +  w \partial_z  V + \nabla_h  p + f_0  V^\perp - \nu_h \Delta_h  V - \nu_z \partial_{zz}  V\right)}_{=0}\Big\|_{H^\ell(\Omega)}\nonumber
    \\
    & \leq  \|\partial_t \widehat V - \partial_t  V\|_{H^\ell(\Omega)} + \|\widehat V \cdot\nabla_h \widehat V -  V \cdot\nabla_h  V\|_{H^\ell(\Omega)} +  \|\widehat w \partial_z \widehat V -  w \partial_z  V\|_{H^\ell(\Omega)} + f_0 \| \widehat V -  V\|_{H^\ell(\Omega)}\nonumber
    \\
    &\qquad + \|\nabla_h \widehat p - \nabla_h p\|_{H^\ell(\Omega)}  + \nu_h  \| \Delta_h \widehat V - \Delta_h V\|_{H^\ell(\Omega)} + \nu_z \| \partial_{zz} \widehat V - \partial_{zz} V\|_{H^\ell(\Omega)}\nonumber
    \\
    & \leq C_{\nu_h,\nu_z,C_{V}, C_{p}, C_{w} , C_{T}}\left(1+ \|V\|_{C^k(\Omega)}  + \|T\|_{C^k(\Omega)} + \|w\|_{C^{k-1}(\Omega)} + \frac{\ln^{k-3}N}{N^2}\right)\frac{1+\ln^{\ell+2}N}{N^{k-\ell -2}}.
\end{align}
Similarly to \eqref{appro-est-pde-1}, one can calculate that
\begin{equation}\label{appro-est-pde-2}
\begin{split}
    &\|\partial_z \widehat p + \widehat T\|_{H^\ell(\Omega)}  \leq \|\partial_z \widehat p - \partial_z p\|_{H^\ell(\Omega)} + \|T-\widehat T\|_{H^\ell(\Omega)} \leq (C_p+ C_T)\frac{1+\ln^{\ell+1}N}{N^{k-\ell -2}},
    \\
    & \|\nabla_h \cdot \widehat V + \partial_z \widehat w\|_{H^\ell(\Omega)} \leq \|V-\widehat V\|_{H^{\ell+1}(\Omega)} + \|w-\widehat w\|_{H^{\ell+1}(\Omega)} \leq (C_{V} + C_w)  \frac{1+\ln^{\ell+1}N}{N^{k-\ell -2}},
\end{split}
\end{equation}
and
\begin{equation}\label{appro-est-pde-3}
    \begin{split}
       &\Big\|\partial_t \widehat T + \widehat V \cdot\nabla_h \widehat T + \widehat w \partial_z \widehat T - \kappa_h \Delta_h \widehat T - \kappa_z \partial_{zz} \widehat T - Q\Big\|_{H^\ell(\Omega)} 
       \\
       \leq & \|\partial_t \widehat T - \partial_t  T\|_{H^\ell(\Omega)} + \|\widehat V \cdot\nabla_h \widehat T -  V \cdot\nabla_h  T\|_{H^\ell(\Omega)} +  \|\widehat w \partial_z \widehat T -  w \partial_z  T\|_{H^\ell(\Omega)} 
    \\
    &+ \kappa_h  \| \Delta_h \widehat T - \Delta_h T\|_{H^\ell(\Omega)} + \kappa_z \| \partial_{zz} \widehat T - \partial_{zz} T\|_{H^\ell(\Omega)}
    \\
    \leq & C_{\kappa_h,\kappa_z,C_{V}, C_{w}, C_{T}}\left(1+ \|V\|_{C^k(\Omega)} +  \|T\|_{C^k(\Omega)} + \|w\|_{C^{k-1}(\Omega)} + \frac{\ln^{k-3}N}{N^2}\right)\frac{1+\ln^{\ell+2}N}{N^{k-\ell -2}}.
    \end{split}
\end{equation}
The combination of \eqref{appro-est-pde-1}--\eqref{appro-est-pde-3} implies that 
\begin{align}\label{Egi-est}
    &\Eg^i[\ell;\theta] = \left[\int_0^{\mathcal T}\left(\|\mathcal R_{i,V}[\theta]\|_{H^{\ell}(\mathcal D)}^2 + \|\mathcal R_{i,P}[\theta]\|_{H^{\ell}(\mathcal D)}^2+\|\mathcal R_{i,T}[\theta]\|_{H^{\ell}(\mathcal D)}^2+\|\mathcal R_{i,div}[\theta]\|_{H^{\ell}(\mathcal D)}^2\right)dt\right]^{\frac12}\nonumber
    \\
    \leq &C_{\nu_h,\nu_z,\kappa_h,\kappa_z,\mathcal T,C_{V}, C_{p}, C_{w} , C_{T}}\left(1+ \|V\|_{C^k(\Omega)}  + \|T\|_{C^k(\Omega)} + \|w\|_{C^{k-1}(\Omega)} + \frac{\ln^{k-3}N}{N^2}\right)\frac{1+\ln^{\ell+2}N}{N^{k-\ell -2}}.
\end{align}

Next we derive the estimate for $\mathcal{E}_G^t[\ell; \theta]$. Note that  $(V,T, \widehat V, \widehat T)\in C^k\left(\Omega\right)$ and
$(w,p, \widehat w, \widehat p)\in C^{k-1}\left(\Omega\right)$. By applying the trace theorem and the fact that $\Omega$ is a Lipschitz domain, for $0\leq \ell\leq k-3$ and $|\alpha|\leq \ell$, we have for $\varphi\in\{V_1, V_2, w, T, p\}$:
\begin{equation*}
\begin{split}
        &\|D^\alpha \varphi - D^\alpha \widehat \varphi\|_{L^{\infty}(\partial \Omega)}\leq C\|D^\alpha \varphi - D^\alpha \widehat \varphi\|_{W^{1,\infty}(\Omega)} \leq C\|\varphi - \widehat \varphi\|_{W^{\ell+1,\infty}(\Omega)}. 
\end{split}
\end{equation*}
Therefore, for $0\leq \ell\leq k-5$, one has
\begin{align}\label{Egt-est}
      \Eg^t[\ell;\theta] &\leq \sum\limits_{i=1}^2\| (V_0)_i - \widehat V_i(t=0)\|_{H^\ell(\mathcal D)} + \| T_0 - \widehat T(t=0)\|_{H^\ell(\mathcal D)}\nonumber
      \\
      &\leq C\sum\limits_{|\alpha|\leq \ell}\left( \sum\limits_{i=1}^2\|D^\alpha V_i - D^\alpha \widehat V_i\|_{L^{\infty}(\partial \Omega)} + \|D^\alpha T - D^\alpha\widehat T\|_{L^{\infty}(\partial \Omega)}\right)\nonumber
      \\
      &\leq C\left( \sum\limits_{i=1}^2\| V_i - \widehat V_i\|_{H^{\ell+3}( \Omega)} + \| T - \widehat T\|_{H^{\ell+3}( \Omega)}\right)
      \leq (C_{V}+C_{T}) \frac{1+\ln^{\ell+3}N}{N^{k-\ell-3}}.
\end{align}

Regarding the estimate for $\mathcal{E}_G^b[\ell; \theta]$, for $|\alpha|\leq \ell$, thanks to the boundary condition \eqref{BC-T3},
\begin{align*}
       &\Big|D^\alpha \widehat V_i(x=1,y,z)- D^\alpha \widehat V_i(x=0,y,z)\Big|  \\
       \leq & \Big|D^\alpha \widehat V_i(x=1,y,z)- D^\alpha  V_i(x=1,y,z)\Big| + \Big|D^\alpha \widehat V_i(x=0,y,z)- D^\alpha V_i(x=0,y,z)\Big|
       \\
       &+ \underbrace{\Big|D^\alpha  V_i(x=1,y,z)- D^\alpha  V_i(x=0,y,z)\Big|}_{=0},
\end{align*}
and therefore,
\begin{align*}
       \left( \int_0^{\mathcal T} \Big|D^\alpha \widehat V_i(x=1,y,z)- D^\alpha \widehat V_i(x=0,y,z)\Big|^2 dt \right)^{\frac12} &\leq C\|D^\alpha V_i -D^\alpha \widehat V_i\|_{L^2(0,\mathcal T; L^{\infty}(\partial \mathcal D))} 
       \\
       &\leq  C_{\mathcal T}\| V_i - \widehat V_i\|_{H^{\ell+3}(\Omega)} \leq C_{V} \frac{1+\ln^{\ell+3}N}{N^{k-\ell-3}}. 
\end{align*}
One can bound other terms similarly, and finally get
\begin{equation}\label{Egb-est}
\begin{split}
      \Eg^b[\ell;\theta] & = \left(\int_0^{\mathcal T}\|\mathcal R_{b}[\ell;
     \theta]\|^2_{L^2(\partial \mathcal D)} dt\right)^{\frac12}
      \leq (C_{V}+ C_w + C_p + C_{T}) \frac{1+\ln^{\ell+3}N}{N^{k-\ell-4}}.
\end{split}
\end{equation}

Finally, we recall that 
\[
\Eg^p[\ell;\theta]^2 = \int_0^{\mathcal T}\left(\|\widehat V\|_{H^{\ell + 3}(\mathcal D)}^2 + \|\widehat p\|_{H^{\ell + 3}(\mathcal D)}^2 + \|\widehat w\|_{H^{\ell + 3}(\mathcal D)}^2 + \|\widehat T\|_{H^{\ell + 3}(\mathcal D)}^2 \right)dt.
\]
Thanks to \eqref{pre-error} we compute, for example, 
\begin{align*}
    \int_0^{\mathcal T} \|\widehat w\|_{H^{\ell + 3}(\mathcal D)}^2 dt \leq &C \int_0^{\mathcal T} (\|w\|_{H^{\ell + 3}(\mathcal D)}^2 + \|w-\widehat w\|_{H^{\ell + 3}(\mathcal D)}^2) dt \leq   C_{\mathcal T}\left(\|w\|_{C^{\ell + 3}(\Omega)}^2 + C_{w} (\frac{1+\ln^{\ell+3}N}{N^{k-\ell-4}})^2\right).
\end{align*}
As the true solution $(V,p,w,T)$ exists globally and 
\[
(V,T)\in C^k\left(\Omega\right), \quad (w,p)\in C^{k-1}\left(\Omega\right),
\]
for any give $\mathcal T>0$ and $m\leq k$, there exists some constant $C_{\mathcal T, m}$ depending only on the time $\mathcal T$ and $m$ such that
\begin{align}\label{sol-uniform-bdd}
    \|V\|_{C^m(\Omega)}^2 +  \|w\|_{C^{m-1}(\Omega)}^2 + \|p\|_{C^{m-1}(\Omega)}^2 + \|T\|_{C^m(\Omega)}^2 \leq C_{\mathcal T, m}.
\end{align}
Indeed, by \cites{ju2006global,ju2020} one has that $C_{\mathcal T, m} = C_m$ independent of time $\mathcal T$ for \ref{Case1}, while analogue result for \ref{Case2} still remains open. Since $\ell+3\leq k-2$, \eqref{sol-uniform-bdd} implies that
\begin{align*}
    \Eg^p[\ell;\theta]^2 \leq C_{\mathcal T, k-1} + C_{\mathcal T} (C_V+C_p+C_w+C_T) (\frac{1+\ln^{\ell+3}N}{N^{k-\ell-4}})^2.
\end{align*}
Now by taking $\lambda \leq \frac{\varepsilon^2}{2C_{\mathcal T, k-1}}$, we obtain that
\begin{align}\label{Egp-est}
    \lambda\Eg^p[\ell;\theta]^2 \leq \frac{\varepsilon^2} 2 + \lambda C_{\mathcal T} (C_V+C_p+C_w+C_T) (\frac{1+\ln^{\ell+3}N}{N^{k-\ell-4}})^2.
\end{align}
Notice that the bounds \eqref{Egi-est}, \eqref{Egt-est} and \eqref{Egb-est} and the second part of \eqref{Egp-est} are independent of the neural networks' parameterized solution $(\widehat V_i, \widehat w, \widehat p, \widehat T)$, one can pick $N$ large enough such that the bounds from \eqref{Egi-est}, \eqref{Egt-est} and \eqref{Egb-est} are bounded by $\frac\varepsilon4$, and the second part from \eqref{Egp-est} is bounded by $\frac{\varepsilon^2}4$, to eventually obtain that
\begin{align*}
    \Eg[\ell;\theta] \leq \varepsilon.
\end{align*}

\end{proof}

\subsection{Total Error Estimates}

This section is dedicated to answer \ref{SubQ2} raised in the introduction: given PINNs $(\widehat V, \widehat w, \widehat p, \widehat T):=( V_\theta,  w_\theta,  p_\theta,  T_\theta)$ with a small generalization error $\Eg[s;\theta]$, is the corresponding total error $\mathcal E[s;\theta]$ also small? 

In the sequel, we only consider \ref{Case2}: $\nu_z=0$ and $\kappa_z=0$. The proof of \ref{Case1} is similar and simpler. To simplify the notation, we drop the $\theta$ dependence in the residuals, for example, writing $\mathcal R_{i,V}$ instead of $\mathcal R_{i,V}[\theta]$, when there is no confusion.

\begin{theorem}[Answer of \ref{SubQ2}]\label{thm:total-error}
Let $d\in\{2,3\}$, $\mathcal T>0$, $k\geq 1$ and $r=k+2$. Suppose that $V_0\in H_e^r(\mathcal D)$, $T_0\in H_o^r(\mathcal D)$, and $Q\in C^{k-1}([0,\mathcal T]; H_o^r(\mathcal D))$, and let $(V,w,p,T)$ be the unique strong solution to system \eqref{PE-system}. Let $(\widehat V, \widehat w, \widehat p, \widehat T)$ be tanh neural networks that approximate $(V, w, p, T)$. Denote by $(V^*, w^*, p^*, T^*)= (V-\widehat V, w-\widehat w, p-\widehat p, T-\widehat T)$. Then for $0\leq s\leq k-1$, the total error satisfies
\begin{equation}\label{eq:estimate_subq2}
\mathcal E[s;\theta]^2 \leq C_{\nu_h,\mathcal T} \left(\Eg^2[s;\theta] (1+\frac1{\sqrt{\lambda}}) + \Eg[s;\theta] C_{\mathcal T,s+1}^{\frac12}\right)\exp\left(C_{\nu_h,\kappa_h} \left(C_{\mathcal T,s+1} + \frac{C\Eg[s;\theta]^2}{\lambda}\right)\right),
\end{equation}
where $\lambda$ appears in the definition of $\Eg$ given in \eqref{generalization-error}, and $C_{\mathcal T,s+1}$ is a constant bound defined in \eqref{thm2:true-sol-bdd}.
\end{theorem}

\begin{remark}
\hfill
\begin{enumerate}
    \item Thanks to the global existence of strong solutions to the PEs with horizontal viscosity (Theorem \ref{thm:higher-regularity}), the constant $C_{\mathcal T,s+1}$ is finite for any $\mathcal T>0$. This is in contrast to the analysis for the $3D$ Navier-Stokes equations \cite{de2022error}, where the existence of the global strong solutions is still unknown, and therefore the total error may become very large at some finite time $\mathcal T>0$ even if the generalization error is arbitrarily small. 
    \item The constants appearing on the right hand side of \eqref{eq:estimate_subq2} are free of the neural network solution $(\widehat V, \widehat w, \widehat p, \widehat T)$. Thus, if aligned with Theorem \ref{thm:generalization-error} and set $\lambda$ such that $\frac{\Eg[s;\theta]}{\sqrt{\lambda}} = \mathcal O(1)$, the total error $\mathcal E[s;\theta]$ is guaranteed to be small as long as $\Eg[s;\theta]$ is sufficiently small. In other words, this estimate is independent of training, thus \emph{a priori}.
    \item In contrast with some previous works \cites{de2022error,mishra2022estimates} that only considered the $L^2$ total error estimates, our approach provides  higher-order total error estimate $\int_0^{\mathcal T} \left(\|V^*(t)\|_{H^s}^2 + \| T^*(t)\|_{H^s}^2\right) dt$. To obtain this estimate, it is necessary to control the residuals in corresponding higher-order Sobolev norms. This is a significant improvement over previous approaches, as the higher-order total error estimate provides more accurate and detailed information about the behavior of the solution, which can be crucial for certain applications such as fluid dynamics or solid mechanic. 
\end{enumerate}
\end{remark}

\begin{proof}[Proof of Theorem \ref{thm:total-error}]
 Let $\mathcal T>0$ be arbitrary but fixed. Since $(V,w,p,T)$ is the unique strong solution to system \eqref{PE-system}, from \eqref{residuals-pde} one has
 \begin{subequations}\label{system-difference}
     \begin{align}
         &\partial_t V^* + V^*\cdot\nabla_h V^* +  V^*\cdot\nabla_h \widehat V + \widehat V\cdot\nabla_h V^* + w^*\partial_z V^* +  w^*\partial_z \widehat V + \widehat w\partial_z V^* \nonumber
         \\
         &\qquad \qquad \qquad \qquad \qquad \qquad \qquad \qquad \qquad \qquad \qquad + f_0 V^{*\perp} - \nu_h \Delta_h V^* + \nabla_h p^* = -\mathcal R_{i,V}, \label{system-difference-1}
         \\
         &\partial_z p^* + T^* = - \mathcal R_{i,p}, \label{system-difference-2}
         \\
         &\nabla_h \cdot V^* + \partial_z w^* = -\mathcal R_{i,div},\label{system-difference-3}
         \\
         &\partial_t T^* + V^*\cdot\nabla_h T^* +  V^*\cdot\nabla_h \widehat T + \widehat V\cdot\nabla_h T^* + w^*\partial_z T^* +  w^*\partial_z \widehat T + \widehat w\partial_z T^*  = -\mathcal R_{i,T}.\label{system-difference-4}
     \end{align}
 \end{subequations}
 Moreover, we know that $(\widehat V, \widehat p, \widehat w, \widehat T)$ are smooth and
 \begin{equation}\label{thm2:true-sol}
 \begin{split}
     &(V^*, V, T^*, T)\in L^\infty(0,\mathcal T; H^r(\mathcal D)) \subset L^\infty(0,\mathcal T; W^{k,\infty}(\mathcal D)), 
     \\
     &(\nabla_h V^*, \nabla_h V, \nabla_h T^*,\nabla_h T, w^*, w)\in L^2(0,\mathcal T; H^r(\mathcal D)) \subset L^2(0,\mathcal T; W^{k,\infty}(\mathcal D)),
 \end{split}
 \end{equation}
 and from \eqref{pressure},
 \begin{equation*}
     (p^*, p) \in L^\infty(0,\mathcal T; H^r(\mathcal D))\subset L^\infty(0,\mathcal T; W^{k,\infty}(\mathcal D)).
 \end{equation*}
Using equation \eqref{w}, we can rewrite $w^\ast$ as,
 \begin{equation*}
     \begin{split}
         w^* = -\widehat w(z=0) - \int_0^z \nabla_h \cdot V^*(\boldsymbol x', \widetilde z) d\widetilde z - \int_0^z (\nabla_h \cdot \widehat V + \partial_z \widehat w)(\boldsymbol x', \widetilde z) d\widetilde z.
     \end{split}
 \end{equation*}
Since $\widehat w$ is smooth enough, one has
 \begin{equation}\label{w*-ine} 
     \|w^*\|_{H^s} \leq C(\|\mathcal R_{b}[s]\|_{L^2(\partial \mathcal D)} + \|\nabla_h V^*\|_{H^s} + \|\mathcal{R}_{i,div}\|_{H^s}).
 \end{equation}
 
 For each fixed $0\leq s \leq k-1$, let $|\alpha|\leq s$ a multi-index. Taking $D^\alpha$ derivative on system \eqref{system-difference}, taking the inner product of \eqref{system-difference-1} with $D^\alpha V^*$ and \eqref{system-difference-4} with $D^\alpha T^*$, then summing over all $|\alpha|\leq s$ gives
 \begin{equation}\label{difference-est-1}
     \begin{split}
         &\frac12 \frac d{dt} \left(\|V^*\|_{H^s}^2 + \| T^*\|_{H^s}^2\right) + \nu_h \|\nabla_h  V^*\|_{H^s}^2 +  \kappa_h \|\nabla_h  T^*\|_{H^s}^2 
        \\
        = &- \sum\limits_{|\alpha|\leq s} \Big(\left\langle D^\alpha(V^*\cdot\nabla_h V^* +  V^*\cdot\nabla_h \widehat V + \widehat V\cdot\nabla_h V^* + w^*\partial_z V^* +  w^*\partial_z \widehat V + \widehat w\partial_z V^*), D^\alpha V^*  \right\rangle 
        \\
        &\phantom{x}\hspace{8ex}- \left\langle D^\alpha(V^*\cdot\nabla_h T^* +  V^*\cdot\nabla_h \widehat T + \widehat V\cdot\nabla_h T^* + w^*\partial_z T^* +  w^*\partial_z \widehat T + \widehat w\partial_z T^*), D^\alpha T^*  \right\rangle
        \\
       &\phantom{x}\hspace{8ex} + \left\langle D^\alpha \nabla_h p^*, D^\alpha V^*  \right\rangle + \left\langle D^\alpha \mathcal R_{i,V}, D^\alpha V^*  \right\rangle + \left\langle D^\alpha \mathcal R_{i,T}, D^\alpha T^*  \right\rangle \Big).
     \end{split}
 \end{equation}
We first estimate the nonlinear terms. Applying the Sobolev inequality, the H\"older inequality, and Young's inequality, and combining with  \eqref{w*-ine}, one has
\begin{equation}\label{difference-est-2}
    \begin{split}
        &\left|\left\langle D^\alpha(V^*\cdot\nabla_h V^* + w^* \partial_z \widehat V + V^*\cdot\nabla_h \widehat V + \widehat V\cdot\nabla_h V^*), D^\alpha V^*  \right\rangle \right| 
        \\
        \leq & C\Big(\|V^*\|_{H^s}\|\nabla_h V^*\|_{W^{s,\infty}} + \|w^*\|_{H^s}\|\partial_z \widehat V\|_{W^{s,\infty}} + \|V^*\|_{H^s}\|\nabla_h \widehat V\|_{W^{s,\infty}} 
         + \|\nabla_h V^*\|_{H^s} \|\widehat V\|_{W^{s,\infty}}\Big)\|V^*\|_{H^s}
        \\
        \leq &  C\|\mathcal R_{b}[s]\|_{L^2(\partial \mathcal D)}^2  + C\|\mathcal{R}_{i,div}\|_{H^s}^2 + C(\|V^*\|_{W^{s+1,\infty}} + \|\widehat V\|_{W^{s+1,\infty}})\|V^*\|_{H^s}^2 + C \|\widehat V\|_{W^{s+1,\infty}} \|\nabla_h V^*\|_{H^s} \|V^*\|_{H^s}
        \\
        \leq &C_{\nu_h}\left(1+\|V^*\|^2_{W^{s+1,\infty}} + \|\widehat V\|^2_{W^{s+1,\infty}} \right)\|V^*\|_{H^s}^2 + \frac{\nu_h}6 \|\nabla_h V^*\|_{H^s}^2 + C\|\mathcal R_{b}[s]\|_{L^2(\partial \mathcal D)}^2  + C\|\mathcal{R}_{i,div}\|_{H^s}^2.
    \end{split}
\end{equation}
Similarly for $T$, one can deduce
\begin{equation}\label{difference-est-3}
    \begin{split}
        &\left|\left\langle D^\alpha(V^*\cdot\nabla_h T^* + w^* \partial_z \widehat T + V^*\cdot\nabla_h \widehat T +  \widehat V \cdot\nabla_h T^* ), D^\alpha T^*  \right\rangle \right| 
        \\
         \leq & C\Big(\|V^*\|_{H^s}\|\nabla_h T^*\|_{W^{s,\infty}} + \|w^*\|_{H^s}\|\partial_z \widehat T\|_{W^{s,\infty}} + \|V^*\|_{H^s}\|\nabla_h \widehat T\|_{W^{s,\infty}} 
         + \|\nabla_h T^*\|_{H^s} \|\widehat V\|_{W^{s,\infty}}\Big)\|T^*\|_{H^s}
        \\
        \leq &C_{\nu_h,\kappa_h}\left(1+\|T^*\|^2_{W^{s+1,\infty}} + \|\widehat T\|^2_{W^{s+1,\infty}} + \|\widehat V\|^2_{W^{s+1,\infty}} \right)(\|V^*\|_{H^s}^2 + \|T^*\|_{H^s}^2) 
        \\
        &+ \frac{\nu_h}6 \|\nabla_h V^*\|_{H^s}^2+ \frac{\kappa_h} 2 \|\nabla_h T^*\|_{H^s}^2 + C\|\mathcal R_{b}[s]\|_{L^2(\partial \mathcal D)}^2  + C\|\mathcal{R}_{i,div}\|_{H^s}^2.
    \end{split}
\end{equation}
For the rest of the nonlinear terms, we provide the details for $s\geq 1$. The case of $s=0$ follows easily. By the triangle inequality, one has
\begin{equation}\label{difference-est-4}
\begin{split}
     &\left|\left\langle D^\alpha (\widehat w \partial_z V^* + w^* \partial_z V^*), D^\alpha V^*  \right\rangle \right| =  \left|\left\langle D^\alpha (w \partial_z V^*), D^\alpha V^*  \right\rangle \right|
     \\
     &\leq \left|\left\langle  D^\alpha ( w \partial_z V^*) -  w \partial_z D^\alpha V^*, D^\alpha V^* \right\rangle \right|  + \left|\left\langle  w \partial_z D^\alpha V^*, D^\alpha V^*  \right\rangle \right| := |I_1| + |I_2|.
\end{split}
\end{equation}
Using the H\"older inequality, one has
\begin{equation}\label{difference-est-5}
        |I_1| = \left| \int_{\mathcal D} \sum\limits_{|\beta|=0,\beta < \alpha}^{s-1} \binom{\alpha}{\beta} D^{\alpha-\beta} w D^\beta \partial_z V^* \cdot \D^\alpha V^* d\boldsymbol x\right| \leq C\| w\|_{W^{s+1,\infty}}\|V^*\|_{H^s}^2 .
\end{equation}
By integration by parts and thanks to the boundary condition of $w$, we have
 \begin{equation}\label{difference-est-6}
         I_2 = 
        - \frac12 \int_{\mathcal D} |D^\alpha V^*|^2  \partial_z  w d\boldsymbol x  \leq C \| w\|_{W^{s+1,\infty}} \|V^*\|_{H^s}^2.
 \end{equation}
 A similar calculation for $T$ yields
 \begin{equation}\label{difference-est-7}
         \left|\left\langle D^\alpha (\widehat w \partial_z T^* + w^* \partial_z T^*), D^\alpha T^*  \right\rangle \right|\leq C\| w\|_{W^{s+1,\infty}}\|T^*\|_{H^s}^2 .
 \end{equation}
 
 Next, by integration by parts, and using \eqref{PE-2} and \eqref{PE-3}, we have
 \begin{align*}
        &\left\langle D^\alpha \nabla_h p^*, D^\alpha V^*  \right\rangle = \int_0^1 \int_{\partial \mathcal M} D^\alpha p^* D^\alpha V^*\cdot \vec{n} d\sigma(\partial \mathcal M) dz - \int_\mathcal D  D^\alpha p^* D^\alpha \nabla_h \cdot V^* d\boldsymbol x
        \\
        &= \int_0^1 \int_{\partial \mathcal M} D^\alpha p^* D^\alpha V^*\cdot \vec{n} d\sigma(\partial \mathcal M) dz + \int_\mathcal D  D^\alpha p^* D^\alpha( \partial_z w^* + \mathcal R_{i,div}) d\boldsymbol x
        \\
        &= \int_{\partial \mathcal D} D^\alpha p^* (D^\alpha V^*, D^\alpha w^*)^\top\cdot \vec{n} d\sigma(\partial \mathcal D) + \int_\mathcal D  D^\alpha p^*  D^\alpha\mathcal R_{i,div} d\boldsymbol x - \int_\mathcal D D^\alpha \partial_z p^* D^\alpha w^* d\boldsymbol x
        \\
        &= \int_{\partial \mathcal D} D^\alpha p^* (D^\alpha V^*, D^\alpha w^*)^\top\cdot \vec{n} d\sigma(\partial \mathcal D) + \int_\mathcal D  D^\alpha p^*  D^\alpha\mathcal R_{i,div} d\boldsymbol x + \int_\mathcal D D^\alpha (T^* + \mathcal R_{i,p}) D^\alpha w^* d\boldsymbol x
        \\
        &:= I_1 + I_2 + I_3.
 \end{align*}
Since the domain $\mathcal D = \mathcal M\times (0,1)$ is Lipschitz and $(D^\alpha p^*, D^\alpha V^*, D^\alpha w^*) \in W^{1,\infty}(\mathcal D)$, the trace theorem and the H\"older inequality yield
\begin{equation}\label{difference-est-8}
    \begin{split}
        I_1 &\leq C\left(\|D^\alpha p^*\|_{L^\infty( \partial \mathcal D)} + \|D^\alpha V^*\|_{L^\infty( \partial \mathcal D)} + \|D^\alpha w^*\|_{L^\infty( \partial \mathcal D)}\right) \|\mathcal R_{b}[s]\|_{L^1(\partial \mathcal D)} 
\\
&\leq C \left(\|D^\alpha p^*\|_{W^{1,\infty}} + \|D^\alpha V^*\|_{W^{1,\infty}} + \|D^\alpha w^*\|_{W^{1,\infty}}\right) \|\mathcal R_{b}[s]\|_{L^2(\partial \mathcal D)}
\\
&\leq C \left(\|p^*\|_{W^{s+1,\infty}} + \|V^*\|_{W^{s+1,\infty}} + \| w^*\|_{W^{s+1,\infty}}\right) \|\mathcal R_{b}[s]\|_{L^2(\partial \mathcal D)}.
    \end{split}
\end{equation}
By the Cauchy–Schwarz inequality, we have
\begin{equation}\label{difference-est-9}
    \begin{split}
        I_2 \leq C \|p^*\|_{H^{s}} \|\mathcal R_{i,div}\|_{H^{s}}.
    \end{split}
\end{equation}
Thanks to \eqref{w*-ine}, by the Cauchy–Schwarz inequality and Young's inequality,
\begin{equation}\label{difference-est-10}
    \begin{split}
        I_3 &\leq \|w^*\|_{H^s}\|\mathcal{R}_{i,p}\|_{H^s}  + \|T^*\|_{H^s}(\|\mathcal R_{b}[s]\|_{L^2(\partial \mathcal D)} + \|\nabla V^*\|_{H^s} + \|\mathcal{R}_{i,div}\|_{H^s})
        \\
        &\leq \|w^*\|_{H^s}\|\mathcal{R}_{i,p}\|_{H^s}+ \|T^*\|_{H^s}(\|\mathcal R_{b}[s]\|_{L^2(\partial \mathcal D)} + \|\mathcal{R}_{i,div}\|_{H^s}) + C_{\nu_h} \|T^*\|^2_{H^s} + \frac{\nu_h}6 \|\nabla V^*\|_{H^s}.
    \end{split}
\end{equation}
For the rest two terms in \eqref{difference-est-1}, applying the Cauchy–Schwarz inequality gives
\begin{equation}\label{difference-est-11}
\begin{split}
    \left\langle D^\alpha \mathcal R_{i,V}, D^\alpha V^*  \right\rangle + \left\langle D^\alpha \mathcal R_{i,T}, D^\alpha T^*  \right\rangle &\leq \|\mathcal R_{i,V}\|_{H^s} \| V^*\|_{H^s} + \|\mathcal R_{i,T}\|_{H^s} \| T^*\|_{H^s}
    \\
    &\leq \|\mathcal R_{i,V}\|_{H^s} \| V^*\|_{W^{k,\infty}} + \|\mathcal R_{i,T}\|_{H^s} \| T^*\|_{W^{k,\infty}}.
\end{split}
\end{equation}

Combining the estimates \eqref{difference-est-1}--\eqref{difference-est-11}, one gets
\begin{equation*}
    \begin{split}
         &\frac d{dt} \left(\|V^*\|_{H^s}^2 + \| T^*\|_{H^s}^2\right) + \nu_h \|\nabla_h  V^*\|_{H^s}^2 +  \kappa_h \|\nabla_h  T^*\|_{H^s}^2 
         \\
         \leq & C_{\nu_h,\kappa_h}\left(1+\|V^*\|^2_{W^{s+1,\infty}}+\|T^*\|^2_{W^{s+1,\infty}} + \|\widehat V\|^2_{W^{s+1,\infty}} + \|\widehat T\|^2_{W^{s+1,\infty}} + \|w\|^2_{W^{s+1,\infty}} \right)(\|V^*\|_{H^s}^2 + \|T^*\|_{H^s}^2)
         \\
         &+  C_{\nu_h}\left(\|\mathcal R_{i,div}\|_{H^s}+ \|\mathcal R_{b}[s]\|_{L^2(\partial \mathcal D)}  + \|\mathcal R_{i,p}\|_{H^s}+ \|p^*\|_{W^{s+1,\infty}} + \|V^*\|_{W^{s+1,\infty}} + \| w^*\|_{W^{s+1,\infty}} + \|T^*\|_{W^{s+1,\infty}} \right)
         \\
         &\qquad\qquad \times \left(\|\mathcal R_{b}[s]\|_{L^2(\partial \mathcal D)} + \|\mathcal R_{i,div}\|_{H^s}
         + \|\mathcal R_{i,V}\|_{H^s} +  \|\mathcal R_{i,T}\|_{H^s} +  \|\mathcal R_{i,p}\|_{H^s}  \right)
         \\
         := & C_1(\|V^*\|_{H^s}^2 + \|T^*\|_{H^s}^2) + C_2\left(\|\mathcal R_{b}[s]\|_{L^2(\partial \mathcal D)} + \|\mathcal R_{i,div}\|_{H^s}
         + \|\mathcal  R_{i,V}\|_{H^s} +  \|\mathcal R_{i,T}\|_{H^s} +  \|\mathcal R_{i,p}\|_{H^s}  \right).
    \end{split}
\end{equation*}
Thanks to Gronwall inequality, we obtain that, for any $t\in[0,\mathcal T]$
\begin{equation*}
       \|V^*(t)\|_{H^s}^2 + \| T^*(t)\|_{H^s}^2 \leq 
        \left(\mathcal E_G^t[s;\theta](t)^2 + \int_0^{t} C_2(\tilde t)(\mathcal E_G^i[s;\theta] + \mathcal E_G^b[s;\theta])(\tilde t) d \tilde t\right) \exp\left(\int_0^{t} C_1(\tilde t)d\tilde t\right).
\end{equation*}
By integrating $t$ over $[0,\mathcal T]$,
\begin{equation*}
    \begin{split}
       \int_0^{\mathcal T} \left(\|V^*(t)\|_{H^s}^2 + \| T^*(t)\|_{H^s}^2\right)dt 
       \leq 
        C_{\mathcal T}\left(\mathcal E_G[s;\theta]^2 + \mathcal E_G[s;\theta] \int_0^{\mathcal T} C_2(t) dt\right) \exp\left(\int_0^{\mathcal T} C_1(t)d t\right).
    \end{split}
\end{equation*}

The derivation till now gives both $C_1$ and $C_2$ depending on the PINNs approximation $(\widehat V, \widehat w, \widehat p, \widehat T)$, thus if they are large the total error may not be under control. To overcome this issue, we next find proper upper bounds for $C_1$ and $C_2$ that are independent of the outputs of the neural networks.

First, we apply triangle inequality to bound
\[
\|\varphi^*\|_{W^{s+1,\infty}} \leq \|\varphi\|_{W^{s+1,\infty}} + \|\widehat\varphi\|_{W^{s+1,\infty}}, \quad \varphi\in\{V,p,w,T\}.
\]
This implies
\begin{align*}
    C_1 \leq C_{\nu_h,\kappa_h}\left(1+\|V\|^2_{W^{s+1,\infty}}+\|T\|^2_{W^{s+1,\infty}} + \|w\|^2_{W^{s+1,\infty}} + \|\widehat V\|^2_{W^{s+1,\infty}} + \|\widehat T\|^2_{W^{s+1,\infty}}  \right).
\end{align*}
As the true solution $V, T, w$ satisfy \eqref{thm2:true-sol}, there exists a bound $C_{\mathcal T,s+1}$ such that
\begin{align}\label{thm2:true-sol-bdd}
   \int_0^{\mathcal T} \left(\|V\|^2_{W^{s+1,\infty}}+\|T\|^2_{W^{s+1,\infty}}+ \|w\|^2_{W^{s+1,\infty}} \right) dt\leq C_{\mathcal T,s+1} .
\end{align}
For the PINNs approximation $\widehat V$ and $\widehat T$, by Sobolev inequality we have
\begin{align*}
    \int_0^{\mathcal T} \left(\|\widehat V\|^2_{W^{s+1,\infty}} + \|\widehat T\|^2_{W^{s+1,\infty}}\right) dt \leq C \int_0^{\mathcal T} \left(\|\widehat V\|^2_{H^{s+3}} + \|\widehat T\|^2_{H^{s+3}}\right) dt \leq C\Eg^p[s;\theta]^2 \leq \frac{C\Eg[s;\theta]^2}{\lambda}.
\end{align*}
Therefore,
\begin{align*}
    \int_0^{\mathcal T} C_1(t) dt \leq C_{\nu_h,\kappa_h} ( C_{\mathcal T,s+1} + \frac{C\Eg[s;\theta]^2}{\lambda}).
\end{align*}
Similarly, we can find that $C_2$ can be bounded as
\begin{align*}
    \int_0^{\mathcal T} C_2(t) dt \leq C_{\mathcal T}  \left(\int_0^{\mathcal T} C^2_2(t) dt\right)^{\frac12} \leq C_{\nu_h,\mathcal T} \left(\Eg[s;\theta] + C_{\mathcal T, s+1}^{\frac12} + \frac{C\Eg[s;\theta]}{\sqrt{\lambda}}\right). 
\end{align*}
This implies that
\begin{align}
    \mathcal E[s;\theta]^2 \leq C_{\nu_h,\mathcal T} \left(\Eg^2[s;\theta] (1+\frac1{\sqrt{\lambda}}) + \Eg[s;\theta] C_{\mathcal T,s+1}^{\frac12}\right)\exp\left(C_{\nu_h,\kappa_h} \left(C_{\mathcal T,s+1} + \frac{C\Eg[s;\theta]^2}{\lambda}\right)\right).
\end{align}

\end{proof}

\subsection{Controlling the Difference between Generalization Error and Training Error}\label{sec:quadrature}
In this section, we will answer question \ref{SubQ3}: given PINNs $(\widehat V, \widehat w, \widehat p, \widehat T):=( V_\theta,  w_\theta,  p_\theta,  T_\theta)$, can the difference between the corresponding generalization error and the training error be made arbitrarily small?

This question can be answered by the well-established results on numerical quadrature rules. Given $\Lambda \subset \mathbb{R}^d$ and $f\in L^1(\Lambda)$, choose some quadrature points $x_m\in \Lambda$ for $1\leq m\leq M$, and quadrature weights $w_m>0$ for $1\leq m\leq M$, and consider the approximation 
    $\frac{1}{M}\sum_{m=1}^M w_m f(x_m) \approx  \int_\Lambda f(x)dx. $
The accuracy of this approximation depends on the chosen quadrature rule, the number of quadrature points $M$ and the regularity of $f$. In the PEs case, the problem is low-dimensional $d\leq 4$, and allows using standard deterministic numerical quadrature points. As in \cites{de2022error}, we consider the midpoint rule: for $l>0$, we partition $\Lambda$ into $M \sim (\frac1l)^d$ cubes of edge length $l$ and we denote by $\{x_m\}_{m=1}^M$ the midpoints of these cubes. The formula and accuracy of the midpoint rule $\mathcal Q_M^\Lambda$ are then given by, 
\begin{equation}\label{eq:quad-error}
    \mathcal Q_M^\Lambda[f] := \frac{1}{M}\sum_{m=1}^Mf(y_m), \qquad \left|\int_\Lambda f(y) dy - \mathcal Q_M^\Lambda[f]\right| \leq C_f M^{-2/d} \leq C_f l^2,
\end{equation}
where $C_f \leq  C\|f\|_{C^2}$. We remark that for high dimensional PDEs, one will need mesh-free methods to sample points in $\Lambda$, such as Monte Carlo, to avoid the curse of dimensionality, for instance, see \cites{sirignano2018dgm}. 

Let us fix the sample sets $\mathcal S=(\mathcal S_i,\mathcal S_t,\mathcal S_b)$ appearing in \eqref{training-error-details} as the midpoints in each corresponding domain, that is, 
\begin{equation}\label{midpoints}
    \begin{split}
        &\mathcal S_i := \{(t_n,x_n,y_n,z_n): (t_n,x_n,y_n,z_n) \text{ are midpoints of cubes of edge length $l_i$ in } [0,\mathcal T]\times \mathcal D \} ,
        \\
        &\mathcal S_t:= \{(x_n,y_n,z_n): (x_n,y_n,z_n) \text{ are midpoints of cubes of edge length $l_t$ in } \mathcal D \},
        \\
        &\mathcal S_b := \{(t_n,x_n,y_n,z_n): (t_n,x_n,y_n,z_n) \text{ are midpoints of cubes of edge length $l_b$ in } [0,\mathcal T]\times \partial\mathcal D \},
    \end{split}
\end{equation}
and the edge lengths and quadrature weights satisfy
\begin{equation}\label{weights}
   l_i \sim | \mathcal S_i|^{-\frac1{d+1}}, \quad l_t \sim | \mathcal S_t|^{-\frac1{d}}, \quad l_b \sim | \mathcal S_b|^{-\frac1{d}}, \quad 
    w_n^i = \frac1{| \mathcal S_i|}, \quad w_n^t = \frac1{| \mathcal S_t|}, \quad w_n^b = \frac1{| \mathcal S_b|}.
\end{equation}
Based on the estimate \eqref{eq:quad-error}, we have the following theorem concerning the control of the generalization error from the training error.
\begin{theorem}[Answer of \ref{SubQ3}]\label{theorem:ge-by-tr}
 Suppose that $k\geq 5$ and $0\leq s\leq k-5$, let $n=\max\{k^2,d(s^2-s+k+d)\}$. Consider the PINNs $( V_\theta,  w_\theta,  p_\theta,  T_\theta)$ with the generalization error $\Eg[s;\theta]$ and the training error $\Et[s;\theta;\mathcal S]$. Then their difference depends on the size of $\mathcal S$: 
 \begin{equation*}
        \Big|\Eg[s;\theta]^2 - \Et[s;\theta;\mathcal S]^2\Big|\leq  C_{i} |\mathcal S_i|^{-\frac2{d+1}} + C_{t} |\mathcal S_t|^{-\frac2{d}} + C_{b} |\mathcal S_b|^{-\frac2{d}},
\end{equation*}
 where $$C_{i} = C_s N^{(3d+s+7)(6s+24)}(\ln N)^{6(s+4)^2} +C_Q + \lambda C_s N^{(3d+s+8)(6s+30)}(\ln N)^{6(s+5)^2}$$
 $$
C_{t} = C_{b} = C_s N^{(3d+s+5)(6s+12)}(\ln N)^{6(s+2)^2}.$$
\end{theorem}
\begin{proof}
 Following from \eqref{eq:quad-error} and the fact that  $( V_\theta,  w_\theta,  p_\theta,  T_\theta)$ are smooth functions, one obtains the estimate
 \begin{align*}
         &\left|\int_0^{\mathcal T}\|\mathcal R_{i,V}[\theta]\|_{H^s}^2 dt  -  \sum\limits_{(t_n,x_n,y_n,z_n)\in \mathcal S_i}\sum\limits_{|\alpha|\leq s}w_n^i|D^\alpha \mathcal R_{i,V}[\theta](t_n,x_n,y_n,z_n)|^2  \right|
         \\
        \leq &\sum\limits_{|\alpha|\leq s} \left|\|D^\alpha \mathcal R_{i,V}[\theta]\|^2_{L^2(\Omega)} - \sum\limits_{(t_n,x_n,y_n,z_n)\in \mathcal S_i}w_n^i|D^\alpha \mathcal R_{i,V}[\theta](t_n,x_n,y_n,z_n)|^2   \right|
        \\
        \leq & C \|\mathcal R_{i,V}[\theta]\|_{C^{2}([0,\mathcal T]; C^{s+2}(\mathcal D))}^2 |\mathcal S_i|^{-\frac2{d+1}}.
 \end{align*}
 Let $\theta \in \Theta_{L, W, R}$ with $L=3$. Note that $n=\max\{k^2,d(s^2-s+k+d)\}$, therefore by Theorem \ref{thm:generalization-error} and Lemma \ref{lemma:appro}, we know that $R=\mathcal O(N\ln(N))$ and $W=\mathcal O(N^{d+1})$. Now by virtue of Lemma \ref{lemma:nn} we can compute that
 \begin{align*}
    &C \|\mathcal R_{i,V}[\theta]\|_{C^{2}([0,\mathcal T]; C^{s+2}(\mathcal D))}^2 |\mathcal S_i|^{-\frac2{d+1}} \leq C_{s} (W^3 R^{s+4})^{6(s+4)} 
    \\
    \leq &C_s (N^{3d+3} N^{s+4}\ln^{s+4}(N))^{6s+24} |\mathcal S_i|^{-\frac2{d+1}} 
    \leq C_s N^{(3d+s+7)(6s+24)}(\ln N)^{6(s+4)^2} |\mathcal S_i|^{-\frac2{d+1}}.
 \end{align*}
Similarly we can obtain the following bounds.
\begin{align*}
        &\Big|\Eg^i[s;\theta]^2 - \Et^i[s;\theta;\mathcal S_i]^2  \Big|\leq   (C_s N^{(3d+s+7)(6s+24)}(\ln N)^{6(s+4)^2} +C_Q) |\mathcal S_i|^{-\frac2{d+1}},
        \\
        &\Big|\Eg^t[s;\theta]^2 - \Et^t[s;\theta;\mathcal S_t]^2 \Big|\leq C_s N^{(3d+s+5)(6s+12)}(\ln N)^{6(s+2)^2}  |\mathcal S_t|^{-\frac2{d}},
        \\
        &\Big|\Eg^b[s;\theta]^2 - \Et^b[s;\theta;\mathcal S_b]^2 \Big|\leq   C_s N^{(3d+s+5)(6s+12)}(\ln N)^{6(s+2)^2}  |\mathcal S_b|^{-\frac2{d}}.
\end{align*}
Finally for the penalty term we have
\begin{align*}
    \Big|\lambda\Eg^p[s;\theta]^2 - \lambda\Et^p[s;\theta;\mathcal S_i]^2  \Big|\leq \lambda C_s N^{(3d+s+8)(6s+30)}(\ln N)^{6(s+5)^2} |\mathcal S_i|^{-\frac2{d+1}}.
\end{align*}
\end{proof}

\begin{remark}\label{rmk:ge-to-te}
    The above result is an \textit{a priori} estimate in the sense that $\Big|\Eg[s;\theta]^2 - \Et[s;\theta;\mathcal S]^2\Big|$ is controlled by a quantity purely depending on the neural network architectures before training.
   
    From the result above we remark that, in order to have a small difference between $\Eg[s;\theta]$ and $\Et[s;\theta]$, one needs to pick sufficiently large sample sets $\mathcal S=(\mathcal S_i,\mathcal S_t, \mathcal S_b)$ at the order of:
    \begin{align}\label{sample-set}
        &|\mathcal S_i| = \mathcal O\left(\max\{\lambda N^{(3d+s+8)(6s+30)(d+1)} (\ln N)^{6(d+1)(s+5)^2}, N^{(3d+s+7)(6s+24)(d+1)} (\ln N)^{6(d+1)(s+4)^2} \}\right), \nonumber
        \\
       & |\mathcal S_b| = |\mathcal S_t| = \mathcal O\left( N^{(3d+s+5)(6s+12)d} (\ln N)^{6d(s+2)^2}\right).
    \end{align}

\end{remark}

\subsection{Answers to \ref{Q1} and \ref{Q2}}
Combining the answers to \ref{SubQ1}--\ref{SubQ3}, i.e., Theorem \ref{thm:generalization-error}--\ref{theorem:ge-by-tr}, we have the following theorem which answers \ref{Q1} and \ref{Q2}.
\begin{theorem}\label{thm:main}
 Let $d\in\{2,3\}$, $\mathcal T>0$, $k\geq 5$, $r> \frac d2 + 2k$, $0\leq s\leq k-5$, $n=\max\{k^2,d(s^2-s+k+d)\}$, and assume that $V_0\in H_e^r(\mathcal D)$, $T_0\in H_o^r(\mathcal D)$, and $Q\in C^{k-1}([0,\mathcal T]; H_o^r(\mathcal D))$. Let $(V,w,p,T)$ be the unique strong solution to system \eqref{PE-system}. 
 \begin{enumerate}
     \item (answer of \ref{Q1}) For every $N>5$, there exist tanh neural networks $(V_\theta,w_\theta,p_\theta,T_\theta)$ with two hidden layers, of widths at most $3\left\lceil\frac{k+n-2}{2}\right\rceil\abs{P_{k-1,d+2}}+d(N-1) + \lceil T(N-1) \rceil$ and $3\left\lceil\frac{d+1+n}{2}\right\rceil\abs{P_{d+2,d+2}}N^{d}\lceil TN \rceil$ such that for every $0\leq s\leq k-5$,
     \begin{equation*}
         \Et[s;\theta;\mathcal S]^2 \leq C\left(\frac{(\lambda +1)(1+\ln^{2s+6} N)}{N^{2k-2s-8}} + \lambda \right) + C_{i} |\mathcal S_i|^{-\frac2{d+1}} + C_{t} |\mathcal S_t|^{-\frac2{d}} + C_{b} |\mathcal S_b|^{-\frac2{d}},
     \end{equation*}
     where the constants $C, C_{i}, C_b, C_t$ are defined based on the constants appearing in Theorems \ref{thm:generalization-error} and \ref{theorem:ge-by-tr}. For arbitrary $\varepsilon>0$, one can first choose $\lambda$ small enough, and then either require $N$ small and $k$ large enough so that $2k-2s-8$ large enough, or $N$ large enough, and $| \mathcal S_i|, | \mathcal S_b|, | \mathcal S_t|$ large enough according to \eqref{sample-set}, so that $\Et[s;\theta;\mathcal S]^2 < \varepsilon$.
     
     \item (answer of \ref{Q2}) For $0\leq s \leq 2k-1$, the total error satisfies 
     \begin{align*}
       \mathcal E[s;\theta]^2 \leq &C_{\nu_h,\mathcal T} \Big((\Et^2[s;\theta] + C_{i} |\mathcal S_i|^{-\frac2{d+1}} + C_{t} |\mathcal S_t|^{-\frac2{d}} + C_{b} |\mathcal S_b|^{-\frac2{d}}) (1+\frac1{\sqrt{\lambda}})
       \\
       &\hspace{1.5cm}+ \left(\Et[s;\theta] + (C_{i} |\mathcal S_i|^{-\frac2{d+1}} + C_{t} |\mathcal S_t|^{-\frac2{d}} + C_{b} |\mathcal S_b|^{-\frac2{d}})^{\frac12}\right) C_{\mathcal T,s+1}^{\frac12}\Big)
       \\
       &\times \exp\left(C_{\nu_h,\kappa_h} \left(C_{\mathcal T,s+1} + \frac{C(\Et[s;\theta]^2+ C_{i} |\mathcal S_i|^{-\frac2{d+1}} + C_{t} |\mathcal S_t|^{-\frac2{d}} + C_{b} |\mathcal S_b|^{-\frac2{d}})}{\lambda}\right)\right),
\end{align*}
where the constants are defined based on the constants appearing in Theorems \ref{thm:total-error} and \ref{theorem:ge-by-tr}. In particular, when the training error is small enough and $C_{i} |\mathcal S_i|^{-\frac2{d+1}} + C_{t} |\mathcal S_t|^{-\frac2{d}} + C_{b} |\mathcal S_b|^{-\frac2{d}}$ is small enough, the total error will be small.
 \end{enumerate}
\end{theorem}

\begin{proof}
 The proof follows from the proof of Theorem \ref{thm:generalization-error}--\ref{theorem:ge-by-tr}.
\end{proof}

\begin{remark}\label{rmk:accuracy} 
Theorem \ref{thm:main} implies that setting $s>0$ in the training error \eqref{training-error} can guarantee the control of the total error \eqref{total-error} under $H^s$ norm. However, it may be unnecessary to use $s>0$ in \eqref{training-error} if one only needs to bound \eqref{total-error} under $L^2$ norm. See more details in the next section on the numerical experiments.
\end{remark}

\subsection{Conclusion and discussion}
In this paper, we investigate the error analysis for PINNs approximating viscous PEs. We answer \ref{Q1} and \ref{Q2} by giving positive results for \ref{SubQ1}--\ref{SubQ3}. In particular, all the estimates we obtain are \textit{a priori} estimates, meaning that they depend only on the PDE solution, the choice of neural network architectures, and the 
sample size in the quadrature approximation, but not on the actual trained network parameters. Such estimates are crucial when designing a model architecture and selecting hyperparameters prior to the training.

Our key step of obtaining such \textit{a priori} estimates is to include a penalty term in the generalization error $\Eg[s;\theta]$ and training error $\Et[s;\theta]$. This idea is inspired by \cite{biswas2022error} in the study of PINNs for $2D$ NSE. In fact, such \textit{a priori} estimates hold for both $2D$ NSE and $3D$ viscous PEs for any time $\mathcal T>0$, but are only true for a short time interval for $3D$ NSE. The main reason is that the global well-posedness of the $3D$ NSE remains open and is one of the most challenging mathematical problems. To bypass this issue and still get \textit{a priori} estimates for $3D$ NSE, \cite{de2022error} instead requires high regularity $H^k$ with $k>6(3d+8)$ for the  initial condition.


Our result is the first one to consider higher order error estimates under mild initial condition $H^k$ with $k \geq 5$. In particular, we prove that one can control the total error $\mathcal E[s;\theta]$ for $s\geq 0$ provided that $\Et[s;\theta]$ is used as the loss function during the training with the same order $s$.

\section{Numerical Experiments}\label{sec:numerical}
For the numerical experiments, we consider the $2D$ case and set $f_0 = 0$. 
System \eqref{PE-system} reduces to 
\begin{subequations}\label{PE-system-2d}
\begin{align}
    & u_t + uu_x + wu_z -\nu_h u_{xx} - \nu_z u_{zz}  +  p_x = 0 , \label{PE2d-1}
    \\
    &\partial_z p + T= 0, \label{PE2d-2}
    \\
   & u_x + w_z =0, \label{PE2d-3}
    \\
    & T_t + uT_x + wT_z -\kappa_h T_{xx} - \kappa_z T_{zz} = Q.
\end{align}
\end{subequations}
We consider the following Taylor-Green vortex as the benchmark on the domain $(x,z) \in [0, 1]^2$ and $t \in [0, 1]$:
    \begin{equation}\label{taylorgreen}
        \begin{cases}
     u = -\sin (2\pi x) \cos (2\pi z) \exp(-4\pi^2(\nu_h+\nu_z)t)
     \\
     w= \cos(2\pi x) \sin(2\pi z) \exp(-4\pi^2(\nu_h+\nu_z)t)
     \\
     p = \frac14 \cos(4\pi x) \exp(-8\pi^2(\nu_h+\nu_z)t) + \frac1{2\pi} \cos(2\pi z) \exp(-4\pi^2\kappa_zt)
     \\
     T = \sin(2\pi z) \exp(-4\pi^2 \kappa_zt)
     \\
     Q = \pi\cos(2\pi x) \sin(4\pi z) \exp(-4\pi^2 (\nu_h+\nu_z + \kappa_z ) t).
    \end{cases}
    \end{equation}

In this section we show the performance of PINNs approximating system \eqref{PE-system-2d} using the benchmark solution \eqref{taylorgreen} with $\nu_z = \nu_h = \kappa_z = \kappa_h = 0.01$ \eqref{Case1} and $\nu_z = \kappa_z = 0$, $\nu_h = \kappa_h = 0.01$ \eqref{Case2}. We will perform the result by setting $s=0$ ($L^2$ residual) or $s=1$ ($H^1$ residual) in the training error \eqref{training-error} with $\lambda = 0$ during training, and compare their $L^2$ and $H^1$ total error \eqref{total-error}. 

\begin{remark}
   The reason to set $\lambda = 0$ is twofold: the evaluation of $H^{s+3}$ will slow down the algorithm due to computing many higher-order derivatives, and we observe in experiments that the total error is already sufficiently small without including the $\Et^p$ term in \eqref{training-error}. In other words, the inclusion of $\Et^p$ is rather technical and mainly aims to provide prior error estimates rather than posterior ones. 
\end{remark}

For the PINNs architecture, we make use of four fully-connected multilayer perceptrons (MLP), one for each of the unknown functions $(u, w, p, T)$, where each MLP consists of 2 hidden layers with 32 neurons per layer. In all cases, the quadrature for minimizing the training error is computed at equally spaced points using the midpoint rule and by taking 5751 points in the interior of the spatial domain, 1024 points at the initial time $t=0$, and 544 points on the spatial boundary. Following from the midpoint rule, the quadrature weights in \eqref{training-error-details} are $w^i_n = 1/5751$, $w^t_n = 1/1024$, and $w^b_n = 1/544$ for all $n$. The learning rate is 1e--4 while using the Adam algorithm for optimizing the training residuals, and the activation function used for all networks is the hyperbolic tangent function.

Figures \ref{fig:full_viscosity} and \ref{fig:horizontal_viscosity} depict the $L^2$ and $H^1$ errors of the PINNs after being trained using the residuals in the cases $s = 0$ and $s = 1$. Since the pressure $p$ is defined up to a constant term, we only plot the $L^2$ error of $\partial_x p_\theta$ and $\partial_z p_\theta$, rather than the $L^2$ or $H^1$ error of $p_\theta$ itself. Tables \ref{table:case1} and \ref{table:case2} give the absolute and relative total error $\mathcal E[s;\theta]$ defined in \eqref{total-error} with $s=0$ or $s=1$ trained using the $L^2$ residuals or $H^1$ residuals. 

As shown in the figures and tables, the $L^2$ errors of the PINNs when trained with the $H^1$ residuals are not significantly improved compared to the ones trained with the $L^2$ residuals. We think this is due to the already-good learning of solutions under $L^2$ residuals, and that the $L^2$ residuals are sufficient to control the $L^2$ error. However, as expected from the analysis in Section \ref{sec:error-estimate}, we observe a noticeably smaller error in the $H^1$ norm (and in the $L^2$ norm for $\partial_x p_\theta$ and $\partial_x p_\theta$) when trained with the $s=1$ residuals, indicating that the derivatives of the unknown functions are learned better under $H^1$ residuals. 

The PINNs and loss functions for the training error were all implemented using the DeepXDE library \cites{lu2021deepxde}. The code for the results in this paper can be found at \url{https://github.com/alanraydan/PINN-PE}.

\begin{figure}
    \centering
    \includegraphics[width=\textwidth]{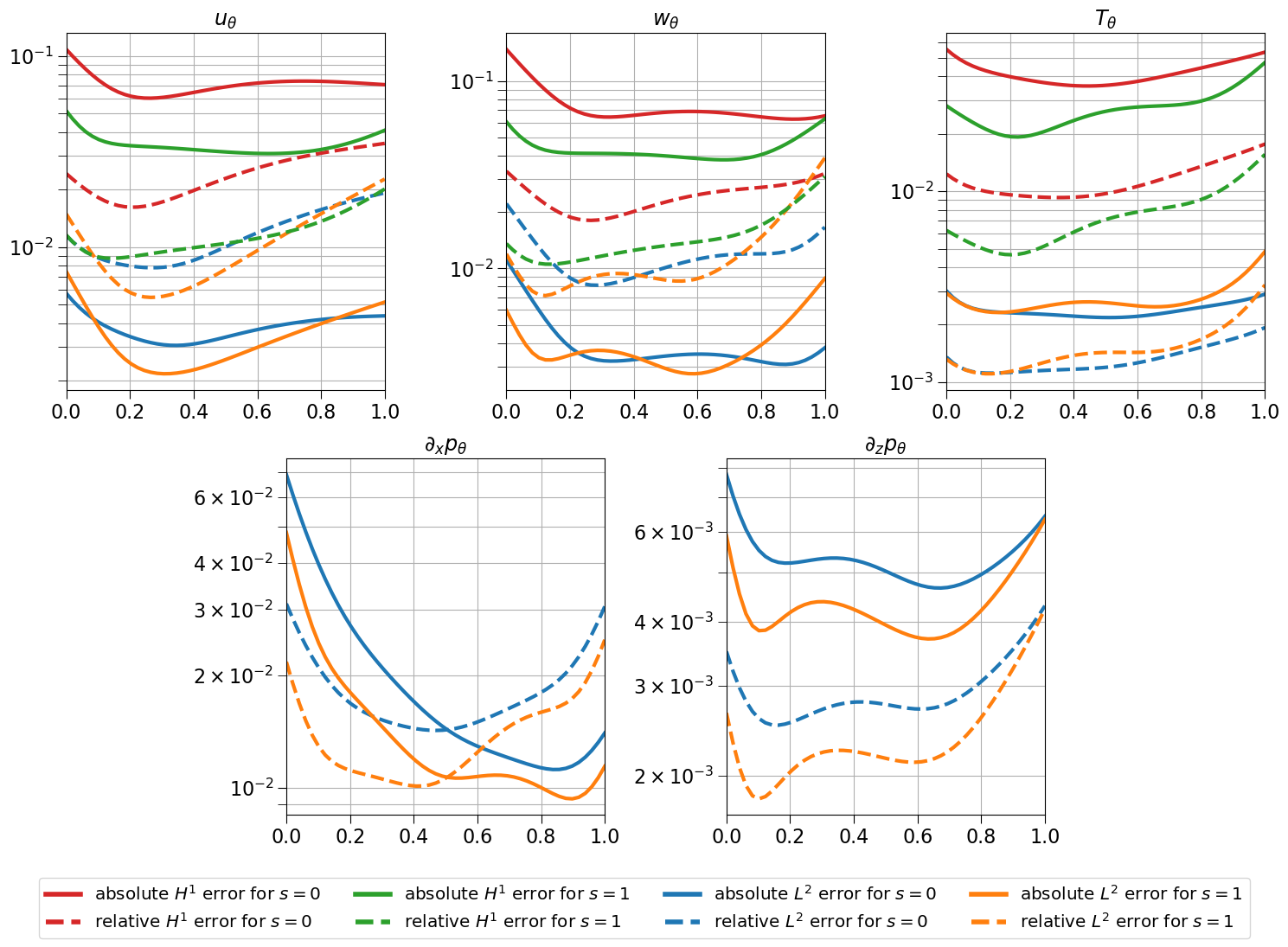}
    \caption{$L^2$ and $H^1$ errors as a function of $t$ between the PINN solutions and the Taylor-Green vortex benchmark \eqref{taylorgreen} with $\nu_z = \nu_h = \kappa_z = \kappa_h = 0.01$ (\ref{Case1}). The notations $s=0$ and $s=1$ represent the $L^2$ residuals and $H^1$ residuals, respectively, in the training.}
    \label{fig:full_viscosity}
\end{figure}

\begin{figure}
    \centering
    \includegraphics[width=\textwidth]{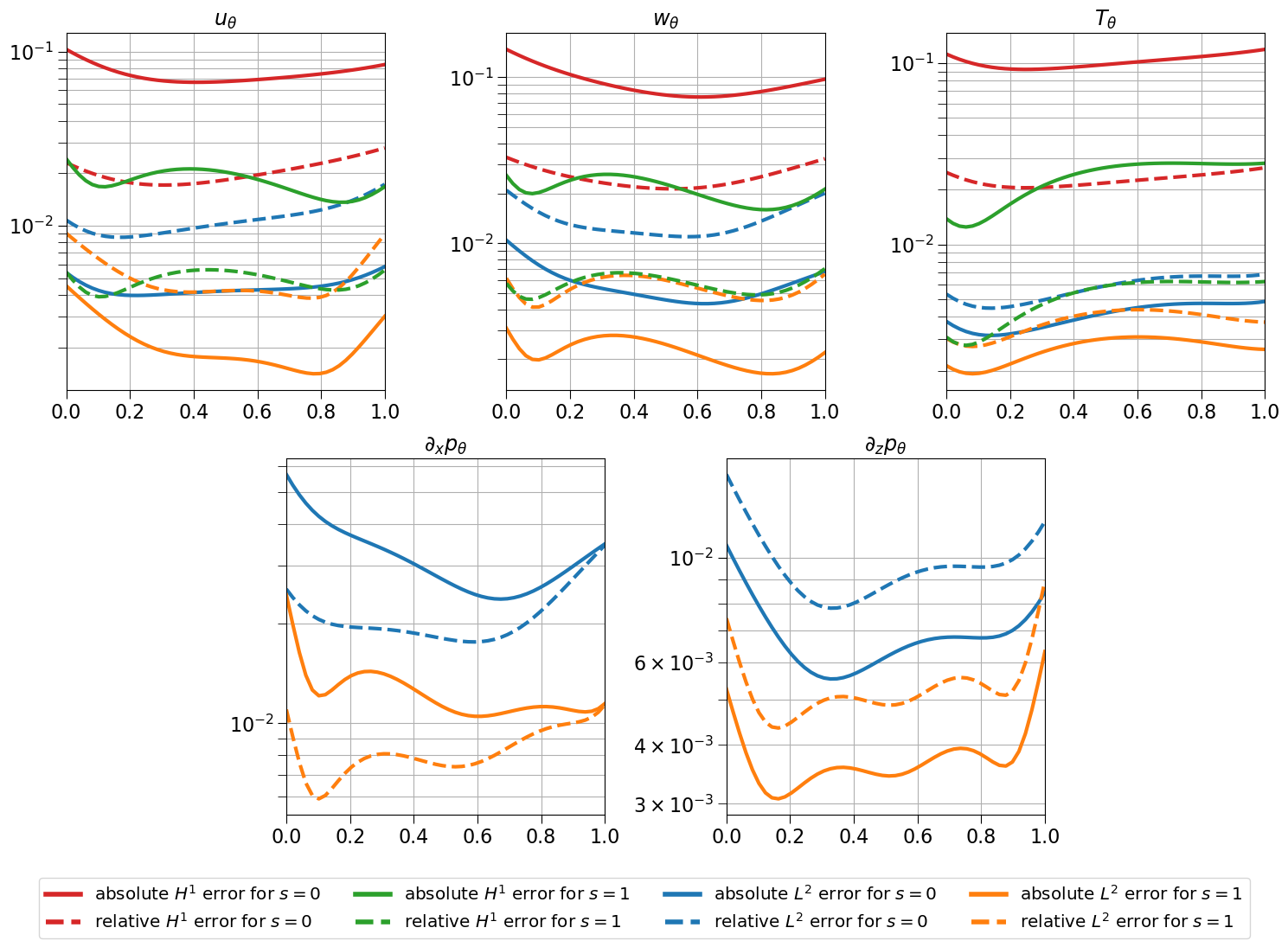}
    \caption{$L^2$ and $H^1$ errors as a function of $t$ between the PINN solutions and the Taylor-Green vortex benchmark \eqref{taylorgreen} with $\nu_z = \kappa_z = 0$, $\nu_h = \kappa_h = 0.01$ (\ref{Case2}). The notations $s=0$ and $s=1$ represent the $L^2$ residuals and $H^1$ residuals, respectively, in the training.}
    \label{fig:horizontal_viscosity}
\end{figure}

\begin{table}
\begin{tabular}{|c||c|c|c|c||}
    \hline &
    \multicolumn{2}{|c|}{$s=0$} & \multicolumn{2}{|c||}{$s=1$} \\
    \hline & $r = 0$  & $r=1$ & $r=0$ & $r=1$ \\
    \hline
    \hline
    absolute & 2.075e--5 & 2.121e-5 &  6.957e--3 & 1.970e--3 \\
    \hline
    relative & 3.021e--5 & 3.089e--5 & 1.419e--3 & 4.017e--4\\
    \hline
\end{tabular}
\caption{\ref{Case1}. Absolute and relative total error $\mathcal E[s;\theta]$ for $s = 0, 1$ for the PINNs trained by minimizing the training error $\mathcal E_T[r; \theta;\mathcal S]$ for $r=0, 1$.}
\label{table:case1}
\end{table}

\begin{table}
\begin{tabular}{|c||c|c|c|c||}
    \hline &
    \multicolumn{2}{|c|}{$s=0$} & \multicolumn{2}{|c||}{$s=1$} \\
    \hline & $r = 0$  & $r=1$ & $r=0$ & $r=1$ \\
    \hline
    \hline
    absolute & 3.669e--5 & 1.222e-5 &  1.614e--2 & 9.022e--4 \\
    \hline
    relative & 4.473e--5 & 1.490e--5 & 2.765e--3 & 1.546e--4 \\
    \hline
\end{tabular}
\caption{\ref{Case2}. Absolute and relative total error $\mathcal E[s;\theta]$ for $s = 0, 1$ for the PINNs trained by minimizing the training error $\mathcal E_T[r; \theta;\mathcal S]$ for $r=0, 1$.}
\label{table:case2}
\end{table}

\section*{Acknowledgments}
Q.L. would like to thank Jinkai Li for interesting discussions on higher-order regularity results for the primitive equations, and is partially supported by Hellman Family Faculty Fellowship. A.R. would like to thank Lu Lu and the DeepXDE maintenance team for providing the PINN implementation library used in this paper. R.H. was partially supported by the NSF grant DMS-1953035, and the Faculty Career Development Award, the Research Assistance Program Award, the Early Career Faculty Acceleration funding and the Regents' Junior Faculty Fellowship at the University of California, Santa Barbara. S.T. was partially supported by  the Regents Junior Faculty fellowship, Faculty Early Career Acceleration grant and Hellman Family Faculty Fellowship sponsored by the University of California Santa Barbara and the NSF under Award No. DMS-2111303. 

Use was made of computational facilities purchased with funds from the National Science Foundation (CNS-1725797) and administered by the Center for Scientific Computing (CSC). The CSC is supported by the California NanoSystems Institute and the Materials Research Science and Engineering Center (MRSEC; NSF DMR 1720256) at UC Santa Barbara.

\appendix
\section{Appendix}
\begin{lemma}[{\cite[Theorem B.7]{de2022error}}]\label{lemma:appro} 
Let $d, n \geq 2, m \geq 3, \delta>0, a_i, b_i \in \mathbb{Z}$ with $a_i<b_i$ for $1 \leq i \leq d, \Omega=\prod_{i=1}^d\left[a_i, b_i\right]$ and $f \in H^m(\Omega)$. Then, for every $N \in \mathbb{N}$ with $N>5$, there exists a tanh neural network $\widehat{f}^N$ with two hidden layers, one of width at most $3\left\lceil\frac{m+n-2}{2}\right\rceil\left|P_{m-1, d+1}\right|+\sum_{i=1}^d\left(b_i-a_i\right)(N-1)$ and another of width at most $3\left\lceil\frac{d+n}{2}\right\rceil\left|P_{d+1, d+1}\right| N^d \prod_{i=1}^d\left(b_i-a_i\right)$, such that for $k \in\{0,1,2,\dots,m-1\}$ it holds that
$$
\left\|f-\widehat{f}^N\right\|_{H^k(\Omega)} \leq C_{k,m,d,f,\delta,\Omega} (1+\ln^k N) N^{-m+k}.
$$
Moreover, the weights of $\hat{f}^N$ scale as $\mathcal O(N\ln(N)+N^\gamma)$ with $\gamma = \max\{m^2, d(k^2-k+m+d) \}/n.$
\end{lemma}

Compared to Theorem B.7 in \cite{de2022error}, here the result holds for $k\in\{0,1,2,\dots,m-1\}$ instead of $k\in\{0,1,2\}$. Notice that $\gamma = \max\{m^2, d(2+m+d) \}/n$ in \cite{de2022error}, which is by taking $k=2$ for the general case. The proof of Lemma \ref{lemma:appro} follows almost the same as Theorem B.7 in \cite{de2022error} and Theorem 5.1 in \cite{de2021approximation}, and the constant $C_{k,m,d,f,\delta,\Omega}$ can be found by following their proofs. We omit the details.

\begin{lemma}[{\cite[Lemma C.1]{de2022error}}]\label{lemma:nn}
Let $d, n, L, W \in \mathbb{N}$, and let $u_\theta: \mathbb{R}^{d+1} \rightarrow \mathbb{R}^{d+1}$ be a neural network with $\theta \in \Theta_{L, W, R}$ for $L \geq 2, R, W \geq 1$; cf. Definition \ref{def:nn}. Assume that $\|\sigma\|_{C^n} \geq 1$. Then it holds for $1 \leq j \leq d+1$ that
$$
\left\|\left(u_\theta\right)_j\right\|_{C^n} \leq 16^L(d+1)^{2 n}\left(e^2 n^4 W^3 R^n\|\sigma\|_{C^n}\right)^{n L} .
$$
\end{lemma}

\bibliographystyle{plain}
\bibliography{Reference}

\end{document}